\documentclass[11pt]{amsart}

\usepackage{xcolor}
\usepackage[all]{xy}

\usepackage{geometry}                
\usepackage[parfill]{parskip}    
\usepackage{graphicx}
\usepackage{amssymb, amsmath, amsthm}
\usepackage{epstopdf}
\usepackage{url}
\usepackage{soul}
\usepackage{calligra}
\usepackage{fancyvrb}
\usepackage{todonotes}
\usepackage[shortlabels]{enumitem}

\usepackage{scalerel}

\usepackage{tikz}\usetikzlibrary{matrix, cd, arrows}
\usepackage{tikz-cd}
\usepackage{xcolor}

\usepackage{tensor}

\usepackage{float}
                                                       
\usepackage[pagebackref]{hyperref}

\usepackage{mathabx,epsfig}

\usepackage{mathscinet}
\usepackage{array}

\DeclareMathAlphabet{\mathscr}{T1}{calligra}{m}{n}
\theoremstyle{plain}
\newtheorem{theorem}{Theorem}[subsection]
\newtheorem{proposition}[theorem]{Proposition}
\newtheorem{corollary}[theorem]{Corollary}
\newtheorem{lemma}[theorem]{Lemma}
\theoremstyle{definition}
\newtheorem{definition}[theorem]{Definition}
\theoremstyle{remark}
\newtheorem{remark}[theorem]{Remark}

\DeclareMathOperator{\codim}{codim}
\newcommand{\End}{\mathcal{E}nd}
\DeclareMathOperator{\Sym}{Sym}

\DeclareMathOperator{\rk}{rk}

\DeclareMathOperator{\Pic}{Pic}

\DeclareMathOperator{\Res}{Res}

\newcommand{\cE}{\mathcal{E}}
\newcommand{\cO}{\mathcal{O}}
\newcommand{\cM}{\mathcal{M}}

\newcommand{\cI}{\mathcal{I}}
\newcommand{\cL}{\mathcal{L}}
\newcommand{\cN}{\mathcal{N}}

\newcommand{\cG}{\mathcal{G}}
\newcommand{\calC}{\mathcal{C}}
\newcommand{\calCt}{\widetilde{\mathcal{C}}}

\newcommand{\PP}{\mathbb{P}}

\newcommand{\calP}{\mathcal{P}}

\newcommand{\GG}{\mathbb{G}}
\newcommand{\LL}{\mathcal{L}}
\newcommand{\MM}{\mathcal{M}}
\newcommand{\MMt}{\widetilde{\mathcal{M}}}
\newcommand{\OO}{\mathcal{O}}

\newcommand{\Ctilda}{{\widetilde{C}}}

\newcommand{\pit}{\widetilde{\pi}}

\newcommand{\GL}{\operatorname{GL}}
\newcommand{\SL}{\operatorname{SL}}

\newcommand{\diffone}{\mathcal{D}^{(1)}}
\newcommand{\difftwo}{\mathcal{D}^{(2)}}

\newcommand{\Hit}{\operatorname{Hit}}
\newcommand{\PH}{\operatorname{PH}}

\newcommand{\circdot}{{_\circ^\circ}}

\newcommand{\rkdetsanti}[1]{\cN^{\pm,s}_{\SL_{#1}}} 

\newcommand{\rkdets}[1]{\MM^s_{\SL_{#1}}}

\newcommand{\rkdetssanti}[1]{\cN^{\pm,ss}_{\SL_{#1}}} 

\newcommand{\rkdetsplus}[1]{\cN^{+,s}_{\SL_{#1}}}

\newcommand{\rkdetsmin}[1]{\cN^{-,s}_{\SL_{#1}}}

\newcommand{\rksanti}[1]{\cN^{\pm,s}_{\GL_{#1}}}

\newcommand{\rksstar}[1]{\MM^{*,s}_{\GL_{#1}}}

\newcommand{\rkssanti}[1]{\cN^{\pm,ss}_{\GL_{#1}}}

\newcommand{\rkstilde}[1]{\MMt^s_{\GL_{#1}}}

\newcommand{\rkdetstilde}[1]{\MMt^s_{\SL_{#1}}} 

\newcommand{\rksstartilde}[1]{\MMt^{*,s}_{\GL_{#1}}}

\newcommand{\rkdetstildepmfixed}[1]{\left(\rkdetstilde{#1}\right)^{\sigma}_{\pm}}

\newcommand{\rkstildepmfixed}[1]{\left(\rkstilde{#1}\right)^{\sigma}_{\pm}}

\newcommand{\rkdetstildeplusfixed}[1]{\left(\rkdetstilde{#1}\right)^{\sigma}_{+}}

\newcommand{\rkdetstildeminusfixed}[1]{\left(\rkdetstilde{#1}\right)^{\sigma}_{-}}

\VerbatimFootnotes

\title{The Prym-Hitchin Connection and Anti-Invariant Level-Rank Duality}
\author{Thomas Baier}
\address{Thomas Baier\\ CAMGSD\\
Instituto Superior T\'ecnico\\
Av. Rovisco Pais\\
1049-001 Lisboa\\
Portugal}
\email{thomas.baier@novasbe.pt}
\author[Michele Bolognesi]{Michele Bolognesi}
\address{Michele Bolognesi\\
Universit\'e Grenoble Alpes\\ CNRS\\
IF\\
38000 Grenoble\\ France}
\email{\mbox{michele.bolognesi@univ-grenoble-alpes.fr}}
\author{Johan Martens}
\address{Johan Martens\\ School of Mathematics and Maxwell Institute\\ The University of Edinburgh\\ Peter Guthrie Tait Road\\ Edinburgh EH9 3FD\\ United Kingdom}
\email{johan.martens@ed.ac.uk}
\author{Christian \textsc{Pauly}}
\address{Christian Pauly \\ Laboratoire de Math\'ematiques J.A. Dieudonn\'e \\ UMR  7351 CNRS \\ Universit\'e C\^ote d'Azur
 \\ 06108 Nice Cedex 02, France}
\email{pauly@unice.fr}

\thanks{
MB was supported by the ANR grant ANR-20-CE40-0023. JM was supported in part by EPSRC grant EP/N029828/1, and would like to thank the Isaac Newton Institute for Mathematical Sciences, Cambridge (funded by EPSRC grant EP/R014604/1), for support and hospitality during the programme \emph{New equivariant methods in algebraic and differential geometry}, where work on this paper was undertaken.  
}

\date{\today}							

\begin{document}

\begin{abstract}
We construct a ``Hitchin-type'' connection on bundles of non-abelian theta functions on  higher-rank Prym varieties, for unramified double covers of curves.  We formulate a version of level-rank duality in this Prym setting (building on work of Zelaci), show it holds for level one, and establish that the duality respects the flat connections at all levels.
\end{abstract}

\maketitle

\section{Introduction}

 Prym varieties are a classical topic of study in algebraic geometry, going back to the 19th century, and brought into a modern context by Mumford \cite{mumford:1974}.  Initially just considered in the context of a double covering of a curve, as the moduli space of line bundles on the covering curve that dualize under the involution, they can be defined for much more general morphisms between curves.  In the original setting of double covers, however, a generalization can also be made to higher-rank vector bundles, as was recently done by Zelaci \cite{ zelaci:thesis, zelaci:twistconf, zelaci:2019, zelaci:2022}. One now considers vector bundles on the covering curve, possibly with fixed determinant, that dualize when pulled back under its involution. 

As with all abelian varieties, the sections of the line bundle that provides the principal polarization (also known as theta functions) for classical Prym varieties give rise to a vector bundle over the moduli space of the abelian varieties.  This bundle is naturally equipped with a flat projective connection  induced by a heat operator for the theta functions \cite{welters:1983}.

For higher-rank vector bundles with fixed determinant over a curve, or principal bundles with a non-abelian, semi-simple structure group, Hitchin constructed a flat projective connection on the associated bundles of non-abelian theta functions \cite{hitchin:1990}.  Similar to Welters' approach, the connection of Hitchin arises through a projective heat operator on the sections of the line bundle over the moduli space of bundles. The symbol of this projective heat operator is dual to the quadratic part of the Hitchin system on the moduli space of Higgs bundles.  Such a symbol uniquely determines a heat operator if a number of cohomological conditions are satisfied, and Hitchin verified these, using the similarities between the Hitchin system and the Narasimhan-Atiyah-Bott K\"ahler form on the moduli space of bundles.

The main aim of this paper is now to construct a ``Hitchin type'' flat projective connection on the bundles of non-abelian theta functions for the higher-rank Prym varieties of Zelaci (Theorems \ref{PHconnection} and \ref{PHFlat}).  

Note that the symplectic geometric (or K\"ahler) description of these moduli spaces (well-known and classical in the standard setting) has not yet been developed.  We therefore take a purely algebro-geometric road to verifying the conditions on the candidate symbol of the heat operator, which in the case of bundles on a curve (no covering) was developed by the authors in \cite{BBMP:2020}.  The key tool there (substituting for the Narasimhan-Atiyah-Bott K\"ahler form) is an explicit determination of the Atiyah class of the determinant-of-cohomology line bundle on the moduli space of bundles, based on work of Beilinson and Schechtman \cite{beilinson.schechtman:1988}. 

Once the existence of the Prym-Hitchin connection is established, we can look at generalizing some of the properties that are known to be satisfied by the classical Hitchin connection. 
In particular, we will look at level-rank duality in the anti-invariant case, extending work of Zelaci \cite{zelaci:twistconf}.  It was shown by Belkale in \cite{belkale:2009} that the level-rank duality respects the connections on the various bundles involved in the statement.  We formulate a version of level-rank duality in the Prym setting, see (\ref{strangedualityatlast}), show that it is an isomorphism at level one (Corollary \ref{anti-lr-level1}), and that at all levels it gives a flat morphism (Theorem \ref{SDflatgeneral}).  We expect the duality to be an isomorphism at all levels, but do not establish this\footnote{Since this paper first appeared on the arXiv, we have been informed by Mukhopadhyay that he has established this result for arbitrary fixed curves \cite{mukh:forth}.}.  In order to do this, we prove in Theorem \ref{equivlaszlo} a variant of a theorem of Laszlo \cite{laszlo:1998}, showing that the Prym-Hitchin connection is equivalent to the WZW connection for twisted conformal blocks.  We also need to make use of conformal embeddings in the setting of twisted conformal blocks (see also \cite{mukhopadhyay.zelaci:2020} for some related work).

The formalism of \cite{BBMP:2020} was recently also used by Biswas, Mukhopadhyay and  Wentworth in \cite{BMW:2023, bmw:2024, biswas2023geometrization} and Ouaras in \cite{ouaras:thesis, ouaras2023parabolic} 
to construct a Hitchin connection in the case of parabolic principal bundles with arbitrary simple structure group, or parabolic vector bundles with arbitrary fixed determinant, respectively.  Both our setting, and that of these parabolic versions, can be understood as special cases of moduli space of torsors for a parahoric Bruhat-Tits group scheme, and we expect a construction of the Hitchin connection to go through in that generality. As of this writing some foundational elements (mainly regarding the corresponding moduli spaces of parahoric Higgs bundles, as well as a parahoric version of Beilinson and Schechtman's theory of trace complexes \cite{beilinson.schechtman:1988}) are missing to carry out the construction in this level of generality.  We hope to revisit this in the future.

In this paper we will only be concerned with unramified covers of curves.  Though much of the results (in particular the existence of the connection) would go through even for ramified covers, and indeed Zelaci's work already takes place in this context, it would add an extra layer of complexity in an already quite baroque setting.  As our own motivating application of the Prym-Hitchin connection does not involve ramification, and ramification is essentially a variation on the parabolic settings mentioned above, we have chosen to focus purely on the new aspects of this work, which is the use of a symmetry that manifests itself simultaneously as the Galois group of a cover of curves, and as automorphisms of the structure group of the bundles in the moduli problem.  This paper is the first to construct a Hitchin connection in this context.

Our primary motivation for developing the Prym-Hitchin connection is that it plays a key role in understanding a sporadic behaviour of the ordinary Hitchin connection at level four, in particular an abelizanization of this connection in terms of theta functions for Prym varieties coming from all unramified double covers of the original curve (which implies that its monodromy is finite, unlike what happens at generic levels).  This result is obtained in our accompanying paper \cite{BBMP:2025}.

The rest of this paper is organised as follows: in Section \ref{backgroundconnections} we recall background material on connections and heat operators.  In Section \ref{higherpryms} we summarise some of the relevant definitions and results from Zelaci's work on higher-rank Prym varieties.  In Section \ref{further} we prove additional results on the cohomology of these higher-rank Prym varieties.  We then show that the requirements for the candidate symbol map to determine a projective heat operator are satisfied, and hence the existence of the Prym-Hitchin connection, in Section \ref{atlast}. In Section \ref{twistedlaszlo} we show that Laszlo's comparison theorem (which shows the equivalence between the Hitchin connection and the WZW connection for bundles of conformal blocks) also holds in the Prym context. In Section \ref{level1} we formulate level-rank 
duality in the Prym setting, and verify that it holds at level one, through considerations of the relevant theta-groups.    The Laszlo theorem then allows us to establish the flatness of the level-rank morphism for general levels, following an approach due to Belkale \cite{belkale:2009}, in Section \ref{generalflat}.  Finally, Appendix \ref{appendixtwistedconf} reviews the relevant constructions for the twisted WZW connection, mainly following Damiolini \cite{damiolini:2020}, and also discusses the consequences of conformal embeddings in this twisted context.

\subsection{Acknowledgments} The authors would like to thank J\o rgen Ellegaard Andersen, 
Prakash Belkale, Indranil Biswas, Chiara Damiolini, Jochen Heinloth, Swarnava Muk\-ho\-padh\-yay, Zakaria Ouaras, Karim Rega, 
Angelo Vistoli, Richard Wentworth 
and Hacen Zelaci for useful conversations and remarks at various stages of this work.

 \section{Connections of Hitchin type}\label{backgroundconnections}

In this section, we recall algebro-geometric conditions for the construction of a projective connection on direct image sheaves following Hitchin's ideas from \cite{BBMP:2020}, to which we refer for the necessary background. For simplicity 
we assume that the base field is $\mathbb{C}$, which will be necessary in sections 7 and 8 dealing with conformal
blocks. As in \cite{BBMP:2020} the construction of the Prym-Hitchin connection remains valid over an algebraically closed base field of arbitrary characteristic (with a few exceptional values).

\subsection{(Projective) connections and Atiyah algebroids}
We briefly recall here A\-ti\-yah's viewpoint on connections \cite{atiyah:1957}, as splitting of Atiyah sequences, as we will be using that throughout.

Given a vector bundle on a smooth scheme $E\rightarrow S$, the \emph{Atiyah algebroid} $\mathcal{A}(E)$ of $E$ is the sheaf of first order differential operators with diagonal symbol, i.e. the middle term in the top short exact sequence (known as the \emph{Atiyah sequence}) of the commutative diagram
$$\begin{tikzcd}
    0\ar[r] & \mathcal{E}nd(E)\ar[r]\ar[d, equal] & \mathcal{A}(E)\ar[r]\ar[d, hook] & T_S\ar[r]\ar[d,hook,"-\otimes \operatorname{Id}_E"] & 0 \\
    0\ar[r] & \mathcal{E}nd(E)\ar[r] & \mathcal{D}^{(1)}_S(E)\ar[r] & T_S\otimes \mathcal{E}nd(E)\ar[r] &0.
\end{tikzcd}$$
A connection $\nabla$ on $E$ is a splitting of the Atiyah sequence
$$\begin{tikzcd} 0\ar[r] & \mathcal{E}nd(E)\ar[r] & \mathcal{A}(E)\ar[r] & T_S\ar[r] \ar[l,dashed, bend left=30, "\nabla"] & 0. \end{tikzcd}$$  It is flat if $\nabla$ preserves the natural Lie brackets.
A projective connection\footnote{This is easily seen to be equivalent to the definition of projective connection given in \cite[Section 1]{looijenga:2013}.} on $E$   is a splitting of the push-out of the Atiyah sequence by $\mathcal{E}nd(E)\rightarrow \mathcal{E}nd(E)\big/\mathcal{O}_S$:
$$\begin{tikzcd} 0\ar[r] & \mathcal{E}nd(E)\big/\mathcal{O}_S\ar[r] & \mathcal{A}(E)\big/\mathcal{O}_S\ar[r] & T_S\ar[r] \ar[l,dashed, bend left=30, "\nabla"] & 0, \end{tikzcd}$$ which is again called flat if it preserves the Lie brackets.  Remark that a (projective) connection on a vector bundle $E$ induces a (projective) connection on the dual bundle $E^*$, and if two bundles carry (projective) connections, there is natural one induced on their tensor product.

Alternatively, one can characterise projective connections through local connections $\nabla^i$ defined on a covering $U_i$ of $S$, with the condition that on the intersections $U_i\cap U_j$ one has that $\nabla^i_X(s)-\nabla^j_X(s)=\omega^{ij}(X)s$, for some $\omega^{ij}\in\Omega^1(U_i\cap U_j)$.

\subsection{Flat morphisms between connections}
A morphism $\Phi$ between bundles $E$ and $F$ over $S$, equipped with connections $\nabla^E$ and $\nabla^F$, is said to be \emph{flat} (or to preserve the connections), if for every open $U
\subset S$, every $X\in T_S(U)$, and every $s\in E(U)$, we have $$\nabla^F_X(\Phi s)=\Phi\left(\nabla^E_X s\right).$$
This is equivalent to $\Phi$, thought of as a section of $F\otimes E^*$, being flat for the tensor connection, i.e. $\nabla^{F\otimes E^*}_X\Phi=0$ for all vector fields $X$.

If $E$ and $F$ are equipped with projective connections, then $\Phi$ is flat if locally $$\widetilde{\nabla}^F_X(\Phi s)-\Phi\left(\widetilde{\nabla}^E_X s\right)=\omega(X) \Phi(s),$$
for some one-form $\omega$, where $\widetilde{\nabla}^F$ and $\widetilde{\nabla}^E$ are local connections lifting the projective connections on $E$ and $F$.
Flat morphisms between bundles with projective connections have constant rank.

\subsection{Projective connections from symbol maps} Consider a smooth surjective morphism of smooth schemes $\pi:\cM\rightarrow S$, and a line bundle $L\rightarrow \cM$ such that $\pi_\ast L$ is 
locally free, hence a vector bundle. One way to construct a connection on $\pi_\ast L$ is via a heat operator on $L$.  We briefly summarise this notion, and refer to \cite{vangeemen.dejong:1998} and \cite{BBMP:2020} for further details. 

On $\mathcal{M}$ 
we can consider the sheaf $$\mathcal{W}_{\cM/S}(L)=\diffone_{\cM}(L)+\difftwo_{\cM/S}(L)\subset \difftwo_{\cM}(L),$$
which fits in the short exact sequence 
$$  \begin{tikzcd}0\ar[r] & \diffone_{\cM/S}(L)\ar[r] & \mathcal{W}_{\cM/S}(L) \ar[r,"^{\sigma_S\oplus\sigma_2}"] &\pi^* (T_S)\oplus \Sym^2 T_{\cM/S}\ar[r] & 0\end{tikzcd}
$$
(here $\difftwo_{\cM/S}(L)$ refers to the second order differential operators that are vertical relative to $\pi$, and $\sigma_S$ and $\sigma_2$ are the natural symbol maps).

A \emph{heat operator} $D$ on $L$ is now a splitting of $\sigma_S$, which has an associated symbol map \begin{equation}\label{symbolofheat}\begin{tikzcd}\rho_D=\pi_{\ast}(\sigma_2)\circ D:  T_S\ \ar[r] &\pi_\ast\Sym^2 T_{\cM/S}.\end{tikzcd}\end{equation}
Similarly a \emph{projective heat operator} $D$ on $L$ is a splitting of $$\begin{tikzcd} \left( \pi_\ast \mathcal{W}_{\cM/S}(L) \right) / \mathcal{O}_S\ar[r,"^{\sigma_S}"]& T_S.\end{tikzcd}$$
Note that a projective heat operator also has a well-defined symbol map as in (\ref{symbolofheat}).

Each (projective) heat operator on $L$ canonically induces a (projective) connection on $\pi_*L$, see \cite[\S 2.3.3]{vangeemen.dejong:1998} or \cite[\S 3.3]{BBMP:2020}.

Moreover, a projective heat operator can in turn be specified by a suitable candidate symbol map.  To this end, let us denote the Kodaira--Spencer map associated to $\pi$ by
\[\begin{tikzcd}
  \kappa_{\cM/S} : T_S \ar[r]& R^1 \pi_\ast T_{\cM/S} ,\end{tikzcd}
\]
and the connecting homomorphism of the long-exact sequence associated to the symbol map of $\pi^{-1}\OO_S$-linear second order differential operators
\[
\begin{tikzcd}
  0 \ar[r] & T_{\cM/S} \ar[r] & \difftwo_{\cM/S}(L) / \OO_{\cM} \ar[r] & \Sym^2 T_{\cM/S} \ar[r] & 0
\end{tikzcd}
\]
by $$\mu_L : \pi_\ast \Sym^2 T_{\cM/S} \to R^1 \pi_\ast T_{\cM/S};$$ also $[L] \in R^1 \pi_\ast \Omega^1_{\cM/S}$ is the Atiyah class of $L$ relative to $\pi$. 
We can now state the conditions for a candidate symbol map to arise from a projective heat operator (this was originally done by Hitchin \cite[Theorem 1.20]{hitchin:1990} in an infinitesimal version). 
\begin{theorem}[{Van Geemen -- De Jong,\cite[\S 2.3.7]{vangeemen.dejong:1998}, \cite[Theorem 3.4.1]{BBMP:2020}}]\label{vgdj} With $L$ and $\pi:\cM\rightarrow S$ as before, 
if the following conditions hold for a given candidate symbol map $\rho: T_S \rightarrow \pi_\ast \Sym^2 T_{\cM/S}$:  \begin{enumerate}[(a)]
\item \label{vgdj-one} $\kappa_{\cM/S}+\mu_{L} \circ \rho=0,$ 
\item \label{vgdj-two} the map $$\begin{tikzcd}\cup [L]: \pi_*T_{\cM/S}\ar[r] &R^1\pi_*\mathcal{O}_{\cM}\end{tikzcd}$$ is an isomorphism, and
\item \label{vgdj-three} $\pi_*\mathcal{O}_{\cM}=\mathcal{O}_S$, 
\end{enumerate} then there exists a unique projective heat operator $D$ on $L$ with symbol $\rho$.
\end{theorem}
Furthermore, flatness of this connection is also related to properties of the symbol map:
\begin{theorem}[{\cite[Theorem 4.9]{hitchin:1990} \cite[Theorem 3.5.1]{BBMP:2020}}]\label{flatness}
Under the conditions of Theorem \ref{vgdj}, the projective connection is flat if
\begin{enumerate}
    \item \label{flatness-one} for all sections $\theta, \theta'$ of $T_S$, we have $\{\rho(\theta), \rho(\theta')\}_{T^*_{\mathcal{M}/S}} = 0$,
    \item \label{flatness-two} the morphism $\mu_L$ is injective, and
    \item \label{flatness-three} there are no vertical vector fields, i.e. $\pi_* T_{\mathcal{M}/S}=0.$
\end{enumerate}
\end{theorem}

\subsection{Hitchin's connection} The main application so far of Theorem \ref{vgdj} is Hitchin's original one to moduli of stable vector bundles with trivial determinant. For referencing purposes we recall the characteristic-free version: consider a smooth family $\pi_{s}:\calC\rightarrow S$ of projective curves of genus $g\geq 2$ (and $g\geq 3$ if $r=2$) defined over an algebraically closed field of characteristic $p$ not dividing $2$ and $r$, and the relative moduli scheme $\pi_e: \cM \to S$ of stable vector bundles of rank $r$ with trivial determinant, with $\LL \to \cM$ the relatively ample generator of its relative Picard group.

Furthermore, let $\begin{tikzcd} \rho^{\Hit} : R^1 {\pi_s}_\ast T_{\calC / S} \ar[r]& {\pi_e}_\ast \Sym^2 T_{\cM / S}\end{tikzcd}$ be the Hitchin symbol, which can be understood as the dual to the quadratic part of the Hitchin system: for a given bundle $E$ of trivial determinant on a curve $C$ we have
\begin{equation*}\begin{tikzcd}
\Sym^2H^0(C,\mathcal{E}nd^0{E})
\ar[r]& H^0(C,K_C^2)
\end{tikzcd}\end{equation*}
which dualizes to 
$$\begin{tikzcd}H^1(C,T_C)\ar[r]& \Sym^2T_{E}\mathcal{M}\cong\Sym^2 H^1(C,\mathcal{E}nd^0(E)),\end{tikzcd}$$
see \cite[Definition 4.3.1]{BBMP:2020} for the global description.

\begin{theorem}[{\cite[Theorems 3.6 and 4.9]{hitchin:1990},\cite[Theorem 4.8.1, 4.8.2]{BBMP:2020}}]
\label{existenceconnection} 
Consider any integer $k$ such that $p$ does not divide $k+r$ and such that $\pi_{e*}(\LL^k)$ is locally free. Then there exists a unique projective connection on the vector bundle $\pi_{e*}(\LL^{k})$ of non-abelian theta functions of level $k$, 
induced by a heat operator with symbol $$\rho=\frac{1}{r+k}\,\left(\rho^{\Hit}\circ \kappa_{\calC/S}\right).$$

Furthermore, the projective connection defined in this way is projectively flat.
\end{theorem}

 \section{Geometry of higher-rank Prym varieties} 
\label{higherpryms}

We now consider an \'etale 
double cover $\calCt$ of a non-singular and projective relative curve $\calC$ over $S$,
\begin{equation}\label{coverofcurves}
  \begin{tikzcd}[column sep=small]
   \calCt \arrow[rr, "p"] \arrow[rd] & & \calC \arrow[ld] \\
   & S &
  \end{tikzcd}
\end{equation}
with corresponding fixed-point free involution $\sigma:\calCt \rightarrow \calCt$ over $S$. It is classical that the Jacobian variety $\mathrm{J}_{\calCt/S}$ of the curve $\calCt$ is isogenous to a product of two abelian varieties, the Jacobian of the quotient $\mathrm{J}_{\calC/S}$ and the \emph{Prym variety} $\mathrm{Pr}_{\calCt/\calC}$. The modern treatment of these classical Prym varieties was initiated by Mumford \cite{mumford:1974} in the investigation of the geometry behind the Schottky--Jung relations.

A natural generalization of this abelian setting occurs for moduli spaces of semi-stable vector bundles of higher rank, and has been studied in the work of Zelaci \cite{zelaci:2022,zelaci:2019,zelaci:thesis}. In this section we recall some of his results and establish several further facts about the geometry of these higher-rank Prym varieties.

\subsection{}\label{BTtorsor} Let $\mathfrak{N}_{\GL_r}\rightarrow S$ be the relative algebraic stack of rank-$r$ bundles $E$ on $\calCt$ equipped with an isomorphism $\psi:E\rightarrow \sigma^*(E^*)$.  This is a smooth stack, locally of finite type.  We will mainly consider two substacks, $\mathfrak{N}^{\pm}_{\GL_r}$, cut out by the extra condition that $\psi=\pm {}^t(\sigma^\ast \psi)$.    Following Zelaci, we refer to these as parameterizing respectively \emph{$\sigma$-symmetric} and \emph{$\sigma$-alternating} bundles (note that the latter can only occur in even rank).  We will use a subscript $\SL_r$ to indicate the stacks where there is a further trivialisation of the determinant $\det(E)\cong \mathcal{O}_{\calCt}$.  

The stacks $\mathfrak{N}^{\pm}_{\GL_r}$ and $\mathfrak{N}^{\pm}_{\SL_r}$ can also be understood as 
a special case of stacks parameterizing \emph{parahoric} torsors, for the parahoric Bruhat-Tits group scheme on $\calC$ obtained as the  $\Gamma$-invariant part of the Weil restriction $\Res_{\calCt/\calC}(G\times \calC)^{\Gamma}$, where $\Gamma=\mathbb{Z}/2\mathbb{Z}$ which acts as the Galois group on $\calCt$, and by the automorphisms 
\begin{equation}\label{syminvol}\Psi^+:g\mapsto {}^{t}(g^{-1})\end{equation}  (for the $\sigma$-symmetric case) and \begin{equation}\label{altinvol}\Psi^-:g\mapsto J{}^{t}(g^{-1})J^{-1}\end{equation} (for the $\sigma$-alternating case )
 
 on $G=\GL_r$ or $\SL_r$, where $J$ is the matrix representing the standard symplectic form on $\mathbb{C}^r$, see 
\cite[\S 2]{zelaci:2019}, \cite[\S 3.1]{zelaci:twistconf}, \cite{zelaci:2022}, \cite{heinloth:2010}, \cite{pappas.rapoport:2010} 
for further details.  Torsors for Bruhat-Tits group schemes were also studied by Balaji and Seshadri in \cite{balaji.seshadri:2015}, but we should remark that they are concerned with the case where the group scheme is generically trivial, which will not be the case for us (on the other hand, we do not consider ramification).

\begin{remark}
    Of course one can equivalently conjugate in (\ref{syminvol}) and (\ref{altinvol}) by other matrices representing non-degenerate symmetric or alternating bilinear forms.  In the symmetric situation this is e.g. the case in the application of the Prym-Hitchin connection in \cite{BBMP:2025}, where for $r=3$ the $\sigma$-symmetric automorphism uses the transpose induced from the identification $\mathbb{C}^3\cong \mathfrak{sl}_2(\mathbb{C})$, and the Killing form on the latter, rather than the standard symmetric bilinear form on $\mathbb{C}^3$. 
\end{remark}

A stability condition for anti-invariant bundles, in terms of isotropic subbundles, was developed in \cite[\S 4]{zelaci:2019}, where also the construction was given of good moduli spaces of semi-stable anti-invariant $\sigma$-symmetric and $\sigma$-alternating bundles. They are also constructed through the general theory of Alper, Halpern-Leistner and Heinloth \cite[\S 8]{alper2022existence}.   We shall denote these moduli spaces (which are normal, projective varieties) by $\rkssanti{r}$ and $\rkdetssanti{r}$, and refer to them as \emph{higher-rank Prym varieties}.  Abusing notation, we will denote the two open
subsets corresponding to anti-invariant bundles whose underlying vector bundle is stable
by $\rksanti{r}\subset \rkssanti{r}$ and $\rkdetsanti{r}\subset \rkdetssanti{r}$. Note that these 
two open subsets are strictly contained in the subsets of stable anti-invariant vector bundles.

We could also consider the involution on the moduli spaces  $\rkstilde{r}$ and $\rkdetstilde{r}$ of stable degree $0$, respectively fixed trivial determinant, bundles on $\calCt$, induced by $\sigma$ and given by $E \mapsto \sigma^*(E^*)$. 
For simplicity we shall also denote this involution by $\sigma$. By definition of the open subsets $\rksanti{r}$ and $\rkdetsanti{r}$ we have natural morphisms induced by forgetting the isomorphism $\psi$
\begin{equation}\label{iotas}\iota_{\mathrm{GL}_r}: \rksanti{r} \rightarrow \rkstilde{r} \qquad \text{and} \qquad \iota_{\mathrm{SL}_r}: \rkdetsanti{r} \rightarrow \rkdetstilde{r}.\end{equation}
The images of these  morphisms $\iota_{\mathrm{GL}_r}$ and $\iota_{\mathrm{SL}_r}$ are clearly the fixed-point loci of the involution $\sigma$ on $\rkstilde{r}$ and $\rkdetstilde{r}$, which 
we denote by $\rkstildepmfixed{r}$ and $\rkdetstildepmfixed{r}$. It was shown by Zelaci \cite[Propositions 4.11 and 4.12]{zelaci:2019} that
$\iota_{\mathrm{GL}_r}$ induces an isomorphism between $\rksanti{r}$ and $\rkstildepmfixed{r}$. In the $\mathrm{SL}_r$-case, it is shown 
that $\iota_{\mathrm{SL}_r}: \rkdetsanti{r} \rightarrow \rkdetstildepmfixed{r}$ is an isomorphism if $r$ is odd, and finite \'etale of degree $2$
if $r$ is even. Note that, since the only automorphisms of stable bundles are given by scaling by a scalar, the only anti-invariant bundles whose underlying bundle is stable have to be $\sigma$-symmetric or $\sigma$-alternating.

\subsection{} Let $\rkdetstilde{r} \rightarrow S$ be the relative moduli space of stable rank-$r$ bundles on $\calCt$ with 
trivial determinant. Using the morphism $\iota_{\SL_r}$ from $\rkdetsanti{r}$ to the fixed-point loci of $\rkdetstilde{r}$ we get a direct sum decomposition 
\begin{equation}\label{split} \iota_{\SL_r}^{*} T_{\rkdetstilde{r}/S}
 = T_{\rkdetsanti{r}/S} \oplus \iota^*_{\SL_r}N_{\iota_{\SL_r}(\rkdetsanti{r})/\rkdetstilde{r}},
 \end{equation}
  where $N_{\iota_{\SL_r}(\rkdetsanti{r})/\rkdetstilde{r}}$
 is the relative normal bundle of $\iota_{\SL_r}(\rkdetsanti{r})$
 in $\rkdetstilde{r}$. This split will allow us to deduce some of the geometry of  $\rkdetsanti{r}$ from that of $\rkdetstilde{r}$.

Some further notation we will need is collected in the following diagram:

\begin{equation}\label{diag_N}
  \begin{tikzcd}
    &  \calCt \times_S \rkdetstilde{r} \ar[rr]  & & \rkdetstilde{r}\ar[dddl, shorten >=1.5ex, "\pit_e"] \\
    \calCt \times_S \rkdetsanti{r} \ar[rr,"\pit_n"]\ar[d,swap, "\pit_w"]\arrow[ur] & & \rkdetsanti{r}\ar[dd, shorten >=1.5ex, swap, "\pit_e"] \ar[ur, "\iota_{\SL_r}"] & \\
    \calCt\ar[dd] \ar[drr, shorten >=1ex, "\pit_s"] & & & \\
    & & S & \\
    \calC \ar[urr,shorten >=1ex, swap, "\pi_s"] & & & \\
  \end{tikzcd} .
\end{equation}

We now summarize some of the relevant geometric properties of the higher-rank Prym varieties
in the fixed-determinant case established by Zelaci. We will state the results for one fixed covering
$p :\widetilde{C} \rightarrow C$, but all results also hold for relative coverings \eqref{coverofcurves}.

\begin{theorem}[Zelaci]\label{ThmZelaci}
Let $\sigma: \widetilde{C} \to \widetilde{C}$ be an involution without fixed points  and let 
$\rkdetstildepmfixed{r}$ be the locus of $\sigma$-symmetric/$\sigma$-alternating 
anti-invariant stable vector bundles of rank $r$ and trivial determinant on $\widetilde{C}$ 
in $\rkdetstilde{r}$. Then
\begin{enumerate}[(a)]
 \item $\rkdetstildeplusfixed{r}$ is smooth, connected (\cite[Theorem 4.16]{zelaci:2022}) 
 and of dimension \linebreak ${\frac{1}{2}(\widetilde{g}-1)(r^2-1)}$, where $\widetilde{g}$ is the genus of 
 $\widetilde{C}$ (\cite[\S 2.5]{zelaci:thesis});
 \item $\rkdetstildeminusfixed{r}$ is smooth, connected and (non-canonically) isomorphic to 
   $\rkdetstildeplusfixed{r}$ if $r$ is even, and empty if $r$ is odd (\cite[\S 4.2.2]{zelaci:2022});
 \item the Picard group of $\rkdetsplus{r}$ (hence also of $\rkdetsmin{r}$ if it is non-empty) is infinite cyclic and generated by the square root $\mathcal{P}_r$ of the restriction of the ample generator $\mathcal{L}_r \in \Pic_{\rkdetstilde{r}}$. The square root 
 $\mathcal{P}_r$ is called the \emph{Pfaffian line bundle} (\cite[\S 4.3, in particular Lemma 4.3.5]{zelaci:thesis} and \cite[Theorem 3]{heinloth:2010}).
 \item the Pfaffian line bundle $\mathcal{P}_r$ descends to $\rkdetstildepmfixed{r}$, i.e., there exists a 
 line bundle $\mathcal{P}$ on $\rkdetstildepmfixed{r}$ such that $\iota^*_{\mathrm{SL}_r}(\mathcal{P}) =
 \mathcal{P}_r$.
\end{enumerate}
\end{theorem}

\begin{remark}
    It is worthwhile to describe the above introduced moduli spaces in the particular case when
    $r = 2$. In fact, a rank-$2$ anti-invariant bundle with trivial determinant is also invariant.
    More precisely, one can show that for an étale cover $p: \Ctilda \rightarrow C$
    with associated $2$-torsion line bundle $\Delta \in \mathrm{Pic}_C[2]$ the pull-back of vector bundles by 
    $p$ induces 
    the following isomorphisms 
    $$
    \begin{tikzcd}    \mathcal{M}_{\mathrm{SL}_2}^{ss} \ar[r,"\sim"] &  \mathcal{N}_{\mathrm{SL}_2}^{+,ss}
        \end{tikzcd} \qquad \text{and} \qquad  \begin{tikzcd}    \mathcal{M}_{\mathrm{SL}_2}^{ss}(\Delta) \ar[r,"\sim"] &  \mathcal{N}_{\mathrm{SL}_2}^{-,ss},
        \end{tikzcd}  $$
       where $\mathcal{M}_{\mathrm{SL}_2}^{ss}(\Delta)$ denotes the moduli space of rank-2 bundles with fixed determinant equal to $\Delta$. Under these isomorphisms the open subsets $\rkdetsanti{2}$ correspond
       to the complement of the strictly semi-stable locus (Kummer variety of $J_C$) plus the 2
       components of the (Kummer variety of) Prym varieties mapped to $\mathcal{M}_{\mathrm{SL}_2}^{ss}$
       (and  $\mathcal{M}_{\mathrm{SL}_2}^{ss}(\Delta)$) by taking the direct image under the covering map $p$. 
       The Pfaffian line bundle $\mathcal{P}_2$ coincides with the classical 
        theta line bundle on $\mathcal{M}_{\mathrm{SL}_2}^{ss}$ and $ \mathcal{M}_{\mathrm{SL}_2}^{ss}(\Delta)$. Finally, the étale degree-$2$ maps $\iota_{\mathrm{SL}_2}: \rkdetsanti{2} \rightarrow \rkdetstildepmfixed{2}$ correspond to the quotient by the involution $E \mapsto E \otimes \Delta$
        acting on $\mathcal{M}_{\mathrm{SL}_2}^{ss}$ and $\mathcal{M}_{\mathrm{SL}_2}^{ss}(\Delta)$. Note that in this
        case there are two line bundles $\mathcal{P}$ verifying $\iota_{\mathrm{SL}_2}^*\mathcal{P} = \mathcal{P}_2$.
\end{remark}

We also will need the following additional properties of these anti-invariant moduli spaces.

\begin{proposition}\label{codimcount} If $r>2$ and $g\geq 2$  (or $r=2$ and $g \geq 3$) 
the codimension  of the closed subvariety
$\rkdetssanti{r} \  \setminus \  \rkdetsanti{r}$ in $\rkdetssanti{r}$ is at least $2$.
\end{proposition}

\begin{proof}
It is easily seen that a closed point in $\rkdetssanti{r} \  \setminus \  \rkdetsanti{r}$ corresponds either to a strictly semi-stable anti-invariant bundle or 
to a direct sum of at least two non-isomorphic stable anti-invariant bundles. In both cases
we can give an upper bound of the dimension of these loci as follows.
\\
According to Zelaci \cite[Lemma 4.3 and 4.4]{zelaci:2019} a closed point corresponding to a 
strictly semi-stable anti-invariant bundle can be represented by a direct sum of the type
$$ F_0 \oplus \bigoplus_{i \in I} \left( H_i \oplus \sigma^* H_i^* \right) \qquad \text{or}
 \qquad \bigoplus_{i \in I} \left( H_i \oplus \sigma^* H_i^* \right), $$
where $F_0$ is a stable anti-invariant bundle of degree $0$, the bundles $H_i$ are stable
(but not necessarily anti-invariant) and $|I| > 0$. It will be enough to do the 
computations for $|I| =1$ --- since for $|I| > 1$ the dimension of the loci are obviously smaller.
We recall the dimensions of the moduli spaces (\cite[Thm. 3.13]{zelaci:2022} and \cite[\S 2.5]{zelaci:thesis})
$$ \dim \rksanti{r} 
 = \frac{1}{2} (\widetilde{g}-1) r^2, \
 \dim \rkdetsanti{r}
 = \frac{1}{2} (\widetilde{g}-1) (r^2-1), \
 \dim \rkstilde{r}
 = (\widetilde{g}-1) r^2 + 1.
$$
We denote $r_0 = \rk F_0$, $r_1 = \rk H_1$ with $r = r_0 + 2r_1$ and $r_0 \geq 0$, $r_1 > 0$. 
Thus the dimension of the strictly semi-stable locus equals (note that we have to substract
$\dim \rksanti{1} = \frac{1}{2} (\widetilde{g}-1)$ as we consider fixed trivial determinant)
$$ \frac{1}{2} (\widetilde{g}-1) r_0^2 + (\widetilde{g}-1) r_1^2 + 1 - \frac{1}{2}(\widetilde{g}-1) =
\frac{1}{2} (\widetilde{g}-1) (r_0^2 + 2 r_1^2 -1) + 1.
$$
So its codimension equals
$$ \frac{1}{2}(\widetilde{g} -1) (2r_1^2 + 4 r_0 r_1) -1 = 2(g-1)r_1(r_1 + 2r_0) -1  $$
which is $\geq 2$ by our assumption on $g$ and $r$.

In the second case, it will be enough to compute the dimension of the locus of anti-invariant bundles 
of the form $F_1 \oplus F_2$ with $F_i \in \rksanti{r_i}$ and $\det (F_1 \oplus F_2) = \mathcal{O}$.
This dimension equals 
$$ \frac{1}{2} (\widetilde{g}-1) r_1^2 + \frac{1}{2} (\widetilde{g}-1) r_2^2 - \frac{1}{2}(\widetilde{g}-1) =
\frac{1}{2} (\widetilde{g}-1)(r_1^2 + r_2^2 -1).$$
So its codimension equals $\frac{1}{2} (\widetilde{g}-1) 2 r_1 r_2 = (\widetilde{g}-1)r_1r_2$ which is $\geq 2$ 
by our assumption on $g$ and $r$.
\end{proof}

\begin{proposition} \label{normalmoduli}
    The moduli spaces $\rkdetssanti{r}$ and $\rkssanti{r}$ are normal varieties.
\end{proposition}

\begin{proof}
By \cite[Proposition 1]{heinloth:2010}, we know that the stacks $\mathfrak{N}^{\pm}_{\GL_r}$ and $\mathfrak{N}^{\pm}_{\SL_r}$ are smooth.  By \cite[Proposition 3.18]{heinloth:2017}, the semi-stable loci inside these stacks is open, hence also smooth, in particular normal.  By \cite[Theorem 8.1]{alper2022existence}, these semi-stable loci have good moduli spaces, $\rkssanti{r}$ and $\rkdetssanti{r}$.  By \cite[Theorem 4.16 (viii)]{alper}, these good moduli spaces are normal.
\end{proof}

\begin{corollary} \label{H0N}
We have $H^0(\rkdetsanti{r}, \mathcal{O}) = H^0\left(\rkdetstildepmfixed{r}, \mathcal{O}\right) = \mathbb{C}$.    
\end{corollary}

\begin{proof}
Since $\rkdetssanti{r}$ is a complete variety, we have $H^0(\rkdetssanti{r}, \mathcal{O}) = \mathbb{C}$. On 
the other hand, since
by Proposition \ref{normalmoduli} $\rkdetssanti{r}$ is normal 
and since by Proposition \ref{codimcount} the codimension of $\rkdetssanti{r} \  \setminus \  \rkdetsanti{r}$ in $\rkdetssanti{r}$ is at least $2$, we deduce by Hartogs's theorem that $H^0(\rkdetsanti{r}, \mathcal{O}) = H^0(\rkdetssanti{r}, \mathcal{O}) = \mathbb{C}$. Finally, since the map $\iota_{\mathrm{SL}_r}: \rkdetsanti{r} \rightarrow \rkdetstildepmfixed{r}$ is either
an isomorphism or a finite étale map of degree $2$ (depending on the parity of $r$), we obtain that
$H^0(\rkdetstildepmfixed{r}, \mathcal{O})$ is either isomorphic to, or a direct summand of
$H^0(\rkdetsanti{r}, \mathcal{O})$, which allows to conclude.
\end{proof}

\begin{corollary}
The line bundle $\calP_r$ extends to the full moduli space $\rkdetssanti{r}$, and the direct images $\pit_{e \ast}\calP_r^k$ have finite rank.
\end{corollary}

\begin{remark}
This can also be seen directly by descending the line bundle from the semi-stable locus on the stack, using \cite[Theorem 10.3]{alper} and a similar reasoning as the discussing of the action of the stabilizer given in \cite[Lemma 6.1]{zelaci:twistconf}.
\end{remark}

We will also need the following fact about the relative canonical bundle of 
$\rkdetsanti{r}$.  
\begin{proposition}
The relative canonical bundle on $\rkdetstilde{r}$ pulls back to twice the relative canonical bundle of $\rkdetsanti{r}$, 
i.e.
\[
\iota^*_{\SL_r} K_{\rkdetstilde{r}/S}
 \cong K_{\rkdetsanti{r}/S}^2 .
\]
\end{proposition}
\begin{proof}
As before, 
we denote by $\End^0(\cE) \rightarrow \calCt \times_S \rkdetstilde{r}$ the universal adjoint bundle. The relative tangent bundle $T_{\rkdetstilde{r}/S}$ of $\rkdetstilde{r} \rightarrow S$ is the first direct image
\[
  T_{\rkdetstilde{r}/S} = R^1 \pit_{n \ast} (\End^0(\cE)) .
\]
Now we restrict $T_{\rkdetstilde{r}/S}$ to a component of the anti-invariant locus $\rkdetsanti{r} \subset \rkdetstilde{r}$. Then we have
\[
 T_{\rkdetstilde{r}/S}\vert_{\rkdetsanti{r}} = R^1 \pit_{n \ast} ( \End^0(\cE) \vert_{\calCt \times_S \rkdetsanti{r}} ) .
\]
For simplicity, we drop the subscript $\calCt \times_S \rkdetsanti{r}$. On the fibered product $\calCt \times_S \rkdetsanti{r}$ we have, after replacing the universal bundle
$\cE$ \'etale-locally by some tensor product with a line bundle coming from $\rkdetsanti{r}$,
an isomorphism
\begin{equation*}\begin{tikzcd} \psi : \sigma^*(\cE) \ar[r,"\sim"]& \cE^\ast.\end{tikzcd}
\end{equation*}
This isomorphism induces a natural $\sigma$-linearization of the bundle $\End(\cE)$
\begin{equation*}\begin{tikzcd}
\sigma^* (\End(\cE))  = \sigma^*(\cE) \otimes\sigma^*(\cE^*)
\ar[rr,"\psi \otimes \sigma^* \psi^{-1}"]& & \End(\cE).\end{tikzcd}\end{equation*}
Since the subbundle $\mathcal{O} \hookrightarrow \End(\cE)$
corresponding to homotheties of $\cE$ is clearly $\sigma$-invariant, we also
obtain a $\sigma$-linearization on $\End^0(\cE)$, i.e. an isomorphism
\begin{equation*}\begin{tikzcd} \sigma^* (\End^0(\cE)) \ar[r,"\sim"]& 
\End^0(\cE). \end{tikzcd}\end{equation*}
The involution $\sigma$ is fixed point free, so the vector bundle $\End^0(\cE)$ 
descends by Kempf's lemma to a vector bundle $\mathcal{F} \to \mathcal{C} \times_S \rkdetsanti{r}$
\[
\begin{tikzcd}
  \End^0(\cE) \cong p^\ast \mathcal{F} \ar[d]  & & \mathcal{F} \ar[d] \\
  \widetilde{\mathcal{C}} \times_S \rkdetsanti{r} \ar[rr, "p"] \ar[rd, swap, "\pit_n"] & & \mathcal{C} \times_S \rkdetsanti{r} \ar[ld, "\overline{\pi}_n"] \\
  & \rkdetsanti{r}. &
\end{tikzcd}
\]
As $p$ is unramified, there is a two-torsion element  $\Delta \in J_{\calC/S}[2]$  in the relative Jacobian such that
\[
  p_\ast \mathcal{O}_{\calCt} =
  \mathcal{O}_{\calC} \oplus \Delta,
\]
and by the projection formula we obtain a decomposition 
\begin{multline*}
  T_{\rkdetstilde{r}/S}\vert_{\rkdetsanti{r}} =
  R^1 \pit_{n \ast} (\End^0(\cE)) = R^1 \pit_{n \ast} (p^*\mathcal{F})\\ =
  R^1 \overline{\pi}_{n \ast} (p_*(p^*\mathcal{F}))  =
  R^1 \overline{\pi}_{n \ast} (\mathcal{F}) \oplus
  R^1 \overline{\pi}_{n \ast} (\mathcal{F}\otimes \Delta).  
\end{multline*}
Here we used the fact that $p$ is finite, hence $R^1p_* = 0$. 

Then by the argument of \cite[\S 3.2]{zelaci:2022} 
we obtain 
identifications
$$ T_{\rkdetsplus{r}}
 =  R^1 \overline{\pi}_{n \ast} (\mathcal{F}), \ 
N_{\rkdetsplus{r}/\rkdetstilde{r}} =  R^1 \overline{\pi}_{n \ast} (\mathcal{F} \otimes \Delta)$$ or 
$$ T_{\cN^{-,s}_{\SL_r}} =  R^1 \overline{\pi}_{n \ast} (\mathcal{F} \otimes \Delta), \ 
N_{\rkdetsmin{r}/\rkdetstilde{r}} =  R^1 \overline{\pi}_{n \ast} (\mathcal{F}),
$$
depending on whether we are on the $\sigma$-symmetric or the $\sigma$-alternating component.
So in either case, in order to conclude the statement of the Proposition, it will be enough
to show that 
$$ \det   R^1 \overline{\pi}_{n \ast} (\mathcal{F}) \cong \det  R^1 \overline{\pi}_{n \ast} (\mathcal{F}
\otimes \Delta).$$
We note that $\overline{\pi}_{n \ast}(\mathcal{F}) \cong \overline{\pi}_{n \ast}(\mathcal{F} \otimes \Delta) = 0$
and therefore $\det R^1 \overline{\pi}_{n \ast} (\mathcal{F}) = \det R^\bullet \overline{\pi}_{n \ast}(\mathcal{F})$
and $\det R^1 \overline{\pi}_\ast(\mathcal{F} \otimes \Delta) = \det R^\bullet \overline{\pi}_*(\mathcal{F}
\otimes \Delta)$, where $\det R^\bullet \overline{\pi}_{n \ast}$ denotes the determinant line bundle of a family
of vector bundles. But now the statement $$ \det   R^\bullet \overline{\pi}_{n \ast} (\mathcal{F}) \cong 
\det  R^\bullet \overline{\pi}_{n \ast} (\mathcal{F} \otimes \Delta)$$ follows from general properties of the
determinant line bundles as $\det(\mathcal{F}) \cong \mathcal{O} \ \text{or} \ \Delta$ and 
$\deg( \Delta \vert_{\calC_s}) = 0$.
\end{proof}

\begin{remark} \label{exprcanpf}
Since $K_{\rkdetstilde{r}/S} \cong \mathcal{L}^{-2r}_r$, where $\mathcal{L}_r$ is the
ample generator of $\mathrm{Pic}_{\rkdetstilde{r}}$, and the relative Picard is torsion-free, we obtain that
\[
K_{\rkdetsanti{r}/S} \cong \mathcal{P}^{-2r}_r \in \mathrm{Pic}_{\rkdetsanti{r}/S}.
\]
\end{remark}

\section{The Prym--Hitchin system}\label{further}

\subsection{The Prym--Hitchin system for a fixed double cover}

We consider an étale double cover $p : \Ctilda \rightarrow C$ with associated $2$-torsion line bundle
$\Delta$.
In this section we will denote for simplicity 
\begin{equation} \label{moduliN}
\mathcal{N} = \rkdetstildeplusfixed{r} \qquad \text{and} \qquad \mathcal{M} = \left(\widetilde{\mathcal{M}}^{s}_{\mathrm{GL}_r}\right)^\sigma_+.
\end{equation}
It will be enough to study the $+$ component, since the $-$ components, if it exists, is isomorphic to $\mathcal{N}$ or $\mathcal{M}$.
Recall that $\mathcal{N}$ and $\mathcal{M}$ are smooth non-complete varieties and their 
closed points correspond to stable
bundles $E$ with trivial determinant or degree $0$ over $\Ctilda$ 
which are anti-invariant, i.e. there exists an isomorphism 
$\psi : \sigma^* E \rightarrow E^*$. 
Zelaci studied the analog of the Hitchin system on the cotangent bundles $T^{\vee}\mathcal{N}$
and $T^{\vee}\mathcal{M}$.
We now recollect the main results obtained in \cite{zelaci:2022}. First we recall that a point in
$T^{\vee}\mathcal{M}$ corresponds to an anti-invariant Higgs bundle $(E, \Phi)$ with 
$\Phi \in H^0(\Ctilda, \mathrm{End}(E)K_{\Ctilda})$ satisfying
$\sigma^* \Phi = \Phi^{t}$. Note that this condition is independent of the choice of the
isomorphism $\psi$. Zelaci constructs the analog of the Hitchin map 
\begin{equation} \label{Hitchinantiinvariant}
\begin{tikzcd} \overline{h} : T^{\vee}\mathcal{M} \ar[r]& W:= \bigoplus_{i=1}^r H^0(C, (K_C\otimes \Delta)^{i}).\end{tikzcd}
\end{equation}
Since $T^{\vee}\mathcal{N} \subset T^{\vee}\mathcal{M}$, we will denote the restriction of $\overline{h}$ to 
$T^{\vee}\mathcal{N}$ by $\overline{h}_0$ and note that
$$\begin{tikzcd} \overline{h}_0 : T^{\vee}\mathcal{N} \ar[r]& W_0:= \bigoplus_{i=2}^r H^0(C, (K_C\otimes \Delta)^{i}) \subset W.\end{tikzcd}$$
In the next section we will construct some extensions of $\overline{h}$ and $\overline{h}_0$ to proper maps.

\medskip

\subsection{Codimension estimates in the moduli space of anti-invariant Higgs bundles}

We will denote by $\mathrm{Higgs}_{\Ctilda}(r)$ and $\mathrm{Higgs}_{\Ctilda}(r)_0$ the coarse moduli space of semistable Higgs bundles $(E,\Phi)$ of degree $0$, respectively with fixed trivial determinant, over $\Ctilda$ equipped with the Hitchin maps
$$ \begin{tikzcd}[column sep=1em]
    h : \mathrm{Higgs}_{\Ctilda}(r) \ar[r, shorten >=-.75ex, shorten < =-.65ex] & B = \bigoplus_{i=1}^r H^0(\Ctilda, K_{\Ctilda}^i), \quad
h_0 : \mathrm{Higgs}_{\Ctilda}(r)_0 \ar[r, shorten >=-.75ex, shorten < =-.65ex] & B_0 = \bigoplus_{i=2}^r H^0(\Ctilda, K_{\Ctilda}^i). \end{tikzcd}
$$
The involution $\sigma$ induces an involution on $\mathrm{Higgs}_{\Ctilda}(r)$, which we also denote by $\sigma$, 
by sending $(E,\Phi)$ to $\sigma.(E,\Phi) = (\sigma^* E^*, \sigma^* \Phi^t)$ and on $B$ using Zelaci's canonical linearization on $K_{\Ctilda}$. Then $h$  and $h_0$ are $\sigma$-equivariant. If we denote by 
$\mathrm{Higgs}^\sigma_{\Ctilda}(r)$ and $\mathrm{Higgs}^\sigma_{\Ctilda}(r)_0$ the fixed-point loci of $\sigma$ in $\mathrm{Higgs}_{\Ctilda}(r)$ and $\mathrm{Higgs}_{\Ctilda}(r)_0$, and by
$B^\sigma = W$  and $B^\sigma_0 = W_0$ the fixed-point loci of $\sigma$ in $B$ and $B_0$
we obtain by restriction morphisms, which we denote by $h^\sigma$ and $h^\sigma_0$
$$
\begin{tikzcd} 
\mathrm{Higgs}_{\Ctilda}(r) \ar[r,"h"] &  B \\
\mathrm{Higgs}^\sigma_{\Ctilda}(r) \ar[r,"h^\sigma"] \ar[u, hook] & W \ar[u, hook] 
\end{tikzcd} 
\qquad \text{and} \qquad
\begin{tikzcd} 
\mathrm{Higgs}_{\Ctilda}(r)_0 \ar[r,"h_0"] &  B_0 \\
\mathrm{Higgs}^\sigma_{\Ctilda}(r)_0 \ar[r,"h^\sigma_0"] \ar[u, hook] & W_0 \ar[u, hook]. 
\end{tikzcd} 
$$
Then, since $h_0$ is $\sigma$-equivariant, for any $v \in W_0$ the fibre $(h^\sigma_0)^{-1}(v)$ equals the fixed-point locus of $\sigma$ in the fiber $h_0^{-1}(v)$.
Consider 
$$W^{sm} \subset W \qquad \text{and} \qquad W^{nod} \subset W$$ 
the open subsets parameterizing spectral covers 
$\pi: X_v \rightarrow C$ which are smooth (hence connected), respectively integral with at most one node. Then obviously we have \break
$W^{sm} \subset W^{nod} \subset W$. Similarly we define $W^{sm}_0 \subset W_0$ and $W^{nod}_0 \subset W_0$ for Higgs bundles with fixed trivial determinant.
We define the preimages in $\mathrm{Higgs}^\sigma_{\Ctilda}(r)$
$$ \mathcal{A}^{sm} = (h^\sigma)^{-1}(W^{sm}) \qquad \text{and} \qquad \mathcal{A}^{nod} = (h^\sigma)^{-1}(W^{nod}), $$
as well as their analogs in $\mathrm{Higgs}^\sigma_{\Ctilda}(r)_0$
$$ \mathcal{A}^{sm}_0 = (h^\sigma_0)^{-1}(W^{sm}_0) \qquad \text{and} \qquad \mathcal{A}^{nod}_0 = (h^\sigma_0)^{-1}(W^{nod}_0). $$
We will need the following 
\begin{lemma} \label{Asmooth}
    The open subsets $\mathcal{A}^{sm}, \mathcal{A}^{nod}, \mathcal{A}^{sm}_0, \mathcal{A}^{nod}_0$  are smooth.
\end{lemma}

\begin{proof}
It will be enough to show smoothness for the open subsets $\mathcal{A}^{nod}$ and $\mathcal{A}^{nod}_0$.
By \cite[Proposition 7.4]{nitsure:1991} the open subset of $\mathrm{Higgs}_{\Ctilda}(r)$ corresponding to stable Higgs bundles is smooth. On the other hand, if $v \in W^{nod}$ the spectral cover $X_v$ is irreducible and reduced, so any Higgs bundle in the
fiber $h^{-1}(v)$ is stable (see e.g. \cite[Remark 1.5]{kouvidakis.pantev:1995}). 
Thus $h^{-1}(v)$ is contained in the smooth locus of $\mathrm{Higgs}_{\Ctilda}(r)$ for any 
$v \in W^{nod}$. Since the fixed-point set of an involution acting on a smooth variety is smooth, we deduce that  
$\mathcal{A}^{nod}$ is smooth. A similar result holds for $\mathcal{A}^{nod}_0$. 
\end{proof}

If we denote by $T^{\vee}\mathcal{M}^{nod}$ and  $T^{\vee}\mathcal{N}^{nod}$ the preimages of $\overline{h}$ and $\overline{h}_0$ over 
$W^{nod}$ and $W^{nod}_0$, then we have the following inclusions
$$
\begin{tikzcd}
 T^{\vee}\mathcal{M}^{nod}  \arrow[dr, "\overline{h}"] \arrow[dd, hook]  & \\
                  &     W^{nod} \\
 \mathcal{A}^{nod} \ar[ur, "h^\sigma"] &                                                           
\end{tikzcd}
\qquad
\begin{tikzcd}
 T^{\vee}\mathcal{N}^{nod}  \arrow[dr, "\overline{h}_0"] \arrow[dd, hook]  & \\
                  &     W^{nod}_0 \\
 \mathcal{A}^{nod}_0 \ar[ur, "h^\sigma_0"] &                                                           
\end{tikzcd}
$$
and similar statements for the loci of smooth spectral covers. Note that $h^\sigma$ and $h^\sigma_0$ are proper maps, since 
$h$ is proper by \cite[Theorem 6.1]{nitsure:1991}.

\medskip

The fibers of $h_0$ and $h_0^\sigma$ can be described as follows. We consider first the case $v \in W_0^{sm}$. If 
$\pi: X_v \rightarrow C$ denotes the smooth degree-$r$ spectral cover over $C$ associated to $v \in W_0^{sm}$  and $\widetilde{X}_v$ denotes 
the fiber product over $C$
$$\begin{tikzcd}
    \widetilde{X}_v    \ar[d,"q"] \ar[r,"\widetilde{\pi}"] & \Ctilda \ar[d,"p"] \\
    X_v \ar[r,"\pi"] & C,
    \end{tikzcd}
$$
then we obtain a non-trivial étale double cover $q : \widetilde{X}_v  \rightarrow X_v$ and the fiber
$h_0^{-1}(v)$ is isomorphic \cite{hitchin:1987} to the kernel of the Norm map 
$$ \begin{tikzcd}[column sep=2em]h_0^{-1}(v) := A_v = \mathrm{ker}(\mathrm{Nm} : \mathrm{Jac}(\widetilde{X}_v) \ar[r] & \mathrm{Jac}(\Ctilda)),\end{tikzcd}$$
which is an abelian variety. The involution $\sigma$ of $\Ctilda$ lifts to an involution on $\widetilde{X}_v$, which we also 
denote by $\sigma$. Then $(h_0^\sigma)^{-1}(v)$ equals the fixed-point locus of the action induced by 
$\sigma$ on $h_0^{-1}(v)$, which
by \cite{zelaci:2022} equals the intersection of $A_v$ with the Prym variety $\mathrm{Prym}(\widetilde{X}_v /X_v)$
$$ (h_0^\sigma)^{-1}(v) := A^\sigma_v = A_v \cap \mathrm{Prym}(\widetilde{X}_v /X_v), $$
which is also an abelian variety.

Next we consider the case $v \in W_0^{nod} \setminus W_0^{sm}$. Then the curve $X_v$ is integral with one node and the étale cover 
$\widetilde{X}_v$ is integral with two nodes, which are interchanged by the involution $\sigma$. The fiber $h_0^{-1}(v)$ is again the 
kernel of the Norm map
$$\begin{tikzcd}[column sep=2em] h_0^{-1}(v) := \widehat{A}_v = \mathrm{ker}(\mathrm{Nm} : \widehat{\mathrm{Jac}}(\widetilde{X}_v) \ar[r] &\mathrm{Jac}(\Ctilda)),\end{tikzcd}$$
where $\widehat{\mathrm{Jac}}(\widetilde{X}_v)$ denotes the compactified Jacobian parameterizing rank-$1$ torsion-free sheaves over $\widetilde{X}_v$
of degree $0$. The structure of  $\widehat{A}_v$ is described e.g. in \cite[Section 1.3]{kouvidakis.pantev:1995}.

Similarly the fiber $(h_0^\sigma)^{-1}(v)$ equals the fixed-point locus of the action induced by $\sigma$ on 
$h_0^{-1}(v) := \widehat{A}_v$, which equals the intersection
\begin{equation} \label{fibercompactifiedPrym}
(h_0^\sigma)^{-1}(v) := \widehat{A}^\sigma_v = \widehat{A}_v \cap \widehat{\mathrm{Prym}}(\widetilde{X}_v /X_v), 
\end{equation}
where $\widehat{\mathrm{Prym}}(\widetilde{X}_v /X_v)$ denotes the compactified Prym variety 
(see e.g \cite[\S 4.1]{laza.sacca.voisin:2017}) of the étale double cover
$q : \widetilde{X}_v \rightarrow X_v$. Using standard techniques one can show the following facts, which we will use in the proof of
Proposition \ref{vanishingH1}.

\begin{proposition} \label{H0andH1compactifiedPrym}
    Exactly as for the abelian varieties $A_v$ and $A^\sigma_v$ associated to smooth spectral covers, we have the following isomorphisms for integral nodal spectral covers
    $$ H^0(\widehat{A}^\sigma_v, T) =  H^0(\widehat{A}_v, T)_+ \qquad \text{and} \qquad
     H^1(\widehat{A}^\sigma_v, \mathcal{O}) =  H^1(\widehat{A}_v, \mathcal{O})_+, $$
     and both spaces have dimension $\dim \widehat{A}_v^\sigma$.
\end{proposition}

We will need the following codimension estimate.

\begin{proposition}
If $r\geq 3$ and $g \geq 3$, then
$$ \codim_{\mathcal{A}^{nod}}(\mathcal{A}^{nod} \setminus T^{\vee}\mathcal{M}^{nod}) \geq 3.$$
\end{proposition}

\begin{proof}
From the construction of $\mathcal{A}^{nod}$ and $T^{\vee}\mathcal{M}^{nod}$ we see that a point
in $\mathcal{A}^{nod}$ corresponds to an anti-invariant Higgs bundle $(E, \Phi)$ having
an associated integral spectral curve with at most one node. So the Higgs bundle $(E, \Phi)$ is stable. The 
subset $\mathcal{A}^{nod} \setminus T^{\vee}\mathcal{M}^{nod}$ corresponds to those stable Higgs bundles
such that $E$ is not stable.

We first compute the dimension of the locus of anti-invariant Higgs bundles $(E,\Phi)$ such 
that $E$ is stricly semi-stable. A general strictly
semistable bundle $E$ can be written as an extension
\begin{equation} \label{exseqE}\begin{tikzcd}
0 \ar[r]& E_1 \ar[r]& E \ar[r]& E_2 \ar[r]& 0 \end{tikzcd}
\end{equation}
with $E_1,E_2$ stable and $\deg E_1 = \deg E_2 = 0$. As $E$ is anti-invariant, there exists an
isomorphism $\begin{tikzcd}
  \psi: \sigma^* E \ar[r,"\sim"] & E^{\vee}
\end{tikzcd}$ 
which induces by composition a morphism
$$ \begin{tikzcd} 
\alpha: \sigma^* E_1 \ar[r, hook] & \sigma^* E \ar[r, "\psi"] & E^{\vee} \ar[r] & E_1^{\vee}.
\end{tikzcd}$$
Since $\sigma^* E_1$ and $E_1^{\vee}$ are two stable degree-$0$ bundles, the morphism $\alpha$
is either $0$ or an isomorphism. We will distinguish these two cases.
\\
Suppose $\alpha$ is an isomorphism. Inverting $\alpha$ will give a splitting of the exact sequence \eqref{exseqE}.
Thus $E = E_1 \oplus E_2$ and both $E_1$ and $E_2$ are anti-invariant. The Higgs field $\Phi \in
H^0(\End(E)\otimes K)$ decomposes as
$$ \Phi = \left(  \begin{array}{cc} \Phi_1 & \beta_{2,1} \\
                  \beta_{1,2} & \Phi_2 \end{array}  \right), $$
with $\Phi_i \in H^0(\End(E_i)\otimes K)$ and $\beta_{i,j} \in \mathrm{Hom}(E_i, E_jK)$. The condition
$\sigma^* \Phi = \Phi^t$ implies that $\sigma^* \Phi_i = \Phi_i^t$ and $\beta_{i,j} = \sigma^* \beta_{j,i}^t$.
Thus the $(E_i, \Phi_i)$ are anti-invariant Higgs bundles with $E_i$ stable. Also 
$\dim \mathrm{Hom}(E_1,E_2 K) = \dim \mathrm{Hom}(E_2,E_1 K) = (\widetilde{g}-1)r_1r_2$ for general $E_i$. Since
$\dim  \mathrm{Higgs}_{\Ctilda}(r)^\sigma = 2 \dim \rksanti{r} = (\widetilde{g}-1)r^2$, the codimension of this
locus equals
$$(\widetilde{g}-1)r^2 - \left( (\widetilde{g}-1)r_1^2 + (\widetilde{g}-1)r_2^2 +  (\widetilde{g}-1)r_1r_2 \right) = (\widetilde{g}-1)r_1r_2. $$
\\
Suppose now that $\alpha = 0$. Then we obtain that $E_2 = \sigma^* E_1^*$ and the extension class $e$ of
$$ \begin{tikzcd}0 \ar[r]& E_1 \ar[r]& E \ar[r]& \sigma^* E_1^{\vee} \ar[r]& 0 \end{tikzcd}$$
in $\mathbb{P} \mathrm{Ext}^1(\sigma^* E_1^{\vee}, E_1) = \mathbb{P} H^1(E_1 \otimes \sigma^* E_1)$ is $\sigma$-invariant.
We observe that $E_1 \otimes \sigma^* E_1 = p^*F_1$ for some rank-$r_1^2$ bundle $F_1$ and that
$H^1(E_1 \otimes \sigma^*E_1) = H^1(F_1) \oplus H^1(F_1 \otimes\Delta)$. Since $H^0(E_1 \otimes \sigma^*E_1) = 0$
by stability of $E_1$ (note that $E_1^{\vee} \neq \sigma^* E_1$ for $E_1$ general), we obtain by Riemann-Roch that both
direct summands have dimension $r_1^2 (g-1) = \frac{1}{2}r_1^2(\widetilde{g}-1)$.

\medskip

As in \cite[p. 372]{hitchin:1990}, we note that by stability of the Higgs bundle $(E, \Phi)$ the dimension of the
space of Higgs fields on $E$ modulo automorphisms of $E$ is given by
$$
\dim H^0(\End(E)\otimes K_{\Ctilda})  - \dim H^0(\End(E)) + 1 =   
  \chi(\Ctilda, \End(E)\otimes K_{\Ctilda}) + 1.
$$
Since $E^{\vee} = \sigma^* E$, the bundle $\End(E)\otimes K_{\Ctilda}$ descends to a bundle $F\otimes K_C$ on $C$ and we obtain 
a decomposition
$\chi(\Ctilda, \End(E)\otimes K_{\Ctilda}) = \chi(C, F\otimes K_C) + \chi(C,F\otimes \Delta \otimes K_C)$. We compute that
$\chi(C, F\otimes K_C) = \chi(C,F\otimes \Delta\otimes K_C) = r^2(g-1)$. Thus, restricting attention to $\sigma$-invariant
Higgs fields, i.e., satisfying $\sigma^* \Phi^t = \Phi$, we obtain that its number of parameters
equals $r^2(g-1) + 1 = 2 r_1^2(\widetilde{g} -1) + 1$.

Putting these estimates together and recalling that $\dim \rkstilde{r_1} = r_1^2(\widetilde{g}-1) + 1$ 
we find the following upper bound for the dimension of this locus
$$ r_1^2(\widetilde{g}-1) + 1 + \frac{1}{2}r_1^2(\widetilde{g}-1) -1 + 2 r_1^2(\widetilde{g}-1) + 1 =
 \frac{7}{2}(\widetilde{g}-1)r_1^2 + 1.$$
Since $\dim \mathrm{Higgs}_{\Ctilda}(r)^\sigma = (\widetilde{g}-1)r^2 = 4(\widetilde{g}-1)r_1^2$ the codimension
of this locus is $\geq \frac{1}{2}(\widetilde{g}-1)r_1^2 - 1$.

Next, we consider the locus of anti-invariant Higgs bundles $(E,\Phi)$ such that $E$ is not semistable. Then a general
non-semistable bundle can be written as an extension 
\begin{equation*} \begin{tikzcd}
0 \ar[r]& E_1 \ar[r]& E \ar[r]& E_2 \ar[r]& 0
\end{tikzcd}
\end{equation*}
with $E_1,E_2$ stable and $\mu = \mu(E_1) > 0 > \mu(E_2)$, $r_1 = \rk(E_1)$. 
As $E$ is anti-invariant, the induced map $\sigma^* E_1 \rightarrow 
E_1^*$ is zero by stability of $E_1$ and the fact that $\mu > 0$. Then we conclude that there is
an isomorphism $E_2 = \sigma^* E_1^*$ and the extension class $e$ of 
$$\begin{tikzcd} 0 \ar[r]& E_1 \ar[r]& E \ar[r]& \sigma^* E_1^{\vee} \ar[r]& 0 \end{tikzcd}$$
in $\mathbb{P} \mathrm{Ext}^1(\sigma^* E_1^*, E_1) = \mathbb{P} H^1(E_1 \otimes \sigma^* E_1)$ is $\sigma$-invariant.
We will now give an upper bound of $\dim H^1(E_1 \otimes \sigma^* E_1)_+$ by adapting the argument of
\cite[page 372]{hitchin:1990} to anti-invariant bundles. We choose an effective divisor $D$ on $C$ of degree $d$
with $d = \lfloor \mu \rfloor + 1$, i.e. the integer $d$ is defined by the inequalities 
$ d > \mu \geq d-1$. If we denote by $L = p^*\mathcal{O}(D)$, we see that the condition on $d$ implies that 
$-\mu > \mu -2d$, hence \break $\mathrm{Hom}(E_1^*, \sigma^*E_1 L^{-1}) = H^0(E_1 \otimes \sigma^*E_1 L^{-1}) = 0$. Thus
we can compute by Riemann-Roch
$$ \dim H^1(E_1 \otimes \sigma^*E_1 L^{-1}) = r_1^2(2d -2\mu) + r_1^2(\widetilde{g}-1).$$
On the other hand $E_1 \otimes \sigma^* E_1$ descends to a bundle $F$ on $C$ and we have the equality
$H^1(E_1 \otimes \sigma^*E_1 L^{-1})_+ = H^1(C, F(-D))$. Thus
$$ \dim H^1(E_1 \otimes \sigma^* E_1 L^{-1})_+ = \frac{1}{2} \dim H^1(E_1 \otimes \sigma^*E_1 L^{-1}) = 
r_1^2(d -\mu) + r_1^2(g-1).$$
Also there is a $\sigma$-equivariant surjective map
$$\begin{tikzcd} H^1(E_1 \otimes \sigma^* E_1 L^{-1}) \ar[r]& H^1(E_1 \otimes \sigma^* E_1)\end{tikzcd}$$
which allows us to give an upper bound
$$ \dim H^1(E_1 \otimes \sigma^* E_1)_+ \leq r_1^2(d -\mu) + r_1^2(g-1) \leq r_1^2 g, $$
where we used $d-\mu \leq 1$.

Finally, by the same argument as before, we compute that for a general anti-invariant non-semistable
bundle $E$ the number of parameters of $\sigma$-invariant Higgs fields on $E$ equals $2r_1^2(\widetilde{g}-1) +1$.

Putting these estimates together, we find the following upper bound for the dimension of this locus
$$ r_1^2(\widetilde{g}-1) + 1 + r_1^2g -1 + 2 r_1^2(\widetilde{g}-1) + 1 =
 3(\widetilde{g}-1)r_1^2 + r_1^2 g + 1.$$
Therefore  the codimension
of this locus is $\geq (g-2)r_1^2 - 1$.

It is clear that if $g \geq 3$ and $r\geq 3$ all three lower bounds for the codimensions are $\geq 3$.
\end{proof}

\begin{corollary} \label{codimensionA}
If $r\geq 3$ and $g \geq 3$, then
$$ \codim_{\mathcal{A}_0^{nod}}(\mathcal{A}_0^{nod} \setminus T^{\vee}\mathcal{N}^{nod}) \geq 3.$$
\end{corollary}
\begin{proof}
We note that the Norm map of $\widetilde{\pi}$ gives a fibration of $\mathcal{A}^{nod}$ over the Prym variety 
$P_0 = \mathrm{Prym}(\Ctilda/C)$ which restricts to the determinant
map $T^{\vee}\mathcal{M}^{nod} \rightarrow \mathcal{M} \rightarrow P_0$. Moreover $P_0$ acts by tensor product 
on $\mathcal{A}^{nod}$ preserving the subvariety
$\mathcal{A}^{nod} \setminus T^{\vee}\mathcal{M}^{nod}$. Thus all the fibers of $\mathcal{A}^{nod} \rightarrow P_0$ and
of $\mathcal{A}^{nod} \setminus T^{\vee}\mathcal{M}^{nod} \rightarrow P_0$ have the same dimension and we can conclude
by restricting the fibration $h^\sigma : \mathcal{A}^{nod} \rightarrow W^{nod}$ to the subspace $W^{nod}_0 \subset
W^{nod}$.
\end{proof}

  \begin{remark}
 Our working definition of a moduli space for anti-invariant Higgs bundles is the fixed-point locus 
  $\mathrm{Higgs}^\sigma_{\Ctilda}(r)_0$ inside the coarse moduli space of semistable Higgs bundles 
  $\mathrm{Higgs}^\sigma_{\Ctilda}(r)$ over $\Ctilda$. This provides us 
  with a quasi-projective variety containing the cotangent bundle $T^{\vee} \mathcal{N}$, but this variety
  is \underline{not} a coarse moduli space for the moduli functor associated to anti-invariant Higgs bundles.
  The question of constructing and studying a moduli stack and a coarse moduli space for semistable 
  anti-invariant Higgs bundles was addressed in the recent paper \cite{rega:2024}.
 \end{remark}
 \subsection{Results on some cohomology spaces of \texorpdfstring{$\mathcal{N}$}{N}}
 We recall from  the previous section that $\mathcal{N}$ denotes $\rkdetstildeplusfixed{r}$
 and $\mathcal{P}$ the (descendent of the) Pfaffian line bundle. In this subsection, we will show that 
 $\mathcal{N}$ satisfies condition \ref{vgdj-two} of the van Geemen-de Jong criterion (Theorem \ref{vgdj}).

 \medskip
 The proof of the following proposition is based on \cite[Theorem 2.2]{singh}, which develops ideas 
 already contained in \cite{hitchin:1990}.

\begin{proposition} \label{cupproductPinjective}
The linear maps induced by cup product with the Atiyah class \break $[\mathcal{P}] \in 
H^1(\mathcal{N}, \Omega_{\mathcal{N}}^1)$
$$\begin{tikzcd}\cup [\mathcal{P}] : H^0(\mathcal{N}, \mathrm{Sym}^m T_{\mathcal{N}}) \ar[r]&  
 H^1(\mathcal{N}, \mathrm{Sym}^{m-1} T_{\mathcal{N}}) \end{tikzcd}$$
 are injective $\forall  m \geq 1$.
\end{proposition}

\begin{proof}
The proof goes exactly as in \cite[Theorem 2.2]{singh} with the additional observation that one does not
require normality (as is assumed in \cite[\S 2]{singh}), since we will restrict ourselves to the smooth family
of abelian varieties $\mathcal{A}_0^{sm}$ introduced above. For the convenience of the reader we will outline the full proof.

\medskip
First we note that there is a  natural isomorphism
$$ H^i(T^{\vee} \mathcal{N}, \mathcal{O}_{T^{\vee} \mathcal{N}}) = \bigoplus_{m \geq 0} H^i(\mathcal{N}, 
\mathrm{Sym}^m T_{\mathcal{N}} ),$$
which corresponds to the $\mathbb{C}^*$-character space decomposition for the natural $\mathbb{C}^*$-action on 
the LHS. We also recall that there is a natural inclusion
$$\begin{tikzcd} T^{\vee}\mathcal{N}^{sm} \ar[r, hook]& \mathcal{A}_0^{sm} \end{tikzcd}$$
over $W^{sm}_0$. Now the restriction of regular functions to the open subset $T^{\vee}\mathcal{N}^{sm} \subset T^{\vee} \mathcal{N}$
gives an injective map $H^0(T^{\vee} \mathcal{N}, \mathcal{O}) \subset H^0(T^{\vee}\mathcal{N}^{sm}, \mathcal{O})$. We  note
that $\mathcal{A}_0^{sm}$ is smooth and by Corollary \ref{codimensionA} 
$\codim_{\mathcal{A}_0^{sm}}(\mathcal{A}_0^{sm} \setminus T^{\vee}\mathcal{N}^{sm}) \geq 2$, so by Hartogs's theorem we have equality $H^0(T^{\vee}\mathcal{N}^{sm}, \mathcal{O}) = H^0(\mathcal{A}_0^{sm}, \mathcal{O})$. Also, the map
$h_0^\sigma : \mathcal{A}_0^{sm}  \rightarrow W^{sm}_0$ is proper with connected fibers, so 
by Zariski's main theorem (see \cite[IV, Cor. 18.12.13]{ega}, or
\cite[Lemma 2.1]{singh}) applied to $h_0^\sigma$, we deduce that $(h_0^\sigma)_* \mathcal{O}_{\mathcal{A}} =  \mathcal{O}$. 
The rest of the argument is identical to the argument in \cite{singh}. 

\medskip
The injectivity of the map $\cup [\mathcal{P}]$ will follow from the commutativity of the following diagram, obtained by restricting $\cup [\eta^* \mathcal{P}]$ to a general fiber of the Hitchin map (here $\eta : T^{vee} \mathcal{N} \rightarrow
\mathcal{N} $ is the projection):
$$
\begin{tikzcd}[row sep=small,column sep=small]
 \ & H^0(T^{\vee} \mathcal{N}, \mathcal{O}) \ar[rr, "{\cup [\eta^* \mathcal{P}]}"] \ar[dl, "{res}", hook] & \ & 
 H^1(T^{\vee} \mathcal{N}, \mathcal{O}) \ar[dr, "{res}" ] & \  \\
H^0(T^{\vee} \mathcal{N}^{sm}, \mathcal{O}) \ar[d, "{\nu}"] & \ &\ &\ &  H^1(T^{\vee} \mathcal{N}^{sm}, \mathcal{O}) 
\ar[dd, "{res_{\overline{h}_0^{-1}(v)}}"] \\
H^0(T^{\vee} \mathcal{N}^{sm}, T^{rel})  \ar[d,  "{res_{\overline{h}_0^{-1}(v)}}"] & \ & \ & \ & \  \\
 H^0(A_v \setminus U_v, T_{A_v}) \ar[r,"{\cong}"]  & H^0(A_v, T_{A_v}) \ar[rr, "{\cup [\eta^* \mathcal{P}_{A_v}]}"] & 
\ & H^1(A_v, \mathcal{O}) \ar[r, "{res}", hook] &  H^1(A_v \setminus U_v, \mathcal{O}).
\end{tikzcd}
$$
The vertical map $\nu : H^0(T^{\vee} \mathcal{N}^{sm}, \mathcal{O}) \rightarrow H^0(T^{\vee} \mathcal{N}^{sm}, T^{rel})$ is defined by associating to a regular section $f \in H^0(T^{\vee} \mathcal{N}^{sm}, \mathcal{O})$ (which is constant along the fibers of \break
$\overline{h}_0 : T^{\vee} \mathcal{N}^{sm} \rightarrow W_0^{sm}$) its Hamiltonian vector field $X_f$, which takes values in the relative tangent sheaf $T^{rel}$ of the fibration $\overline{h}_0$.
We observe that $\ker \nu$ consists of constant functions. Consider for $m \geq 1$ a non-zero function 
$f_m \in H^0(\mathcal{N}, \mathrm{Sym}^m T_{\mathcal{N}} ) \subset H^0(T^{\vee} \mathcal{N}, \mathcal{O})$. Since $f_m$ is
non-constant, the $1$-form $df_m$ on $W_0^{sm}$ is non-zero, hence for a general $v \in W_0^{sm}$ we have 
$df_m(v) \neq 0$. Thus we obtain a non-zero element in $H^0(A_v \setminus U_v, T_{A_v})$, where $A_v \setminus U_v 
= \overline{h}_0^{-1}(v)$. Hartogs's theorem implies that $H^0(A_v, T_{A_v}) \cong H^0(A_v \setminus U_v, T_{A_v})$ and 
$H^1(A_v, \mathcal{O}) \hookrightarrow H^1(A_v \setminus U_v, \mathcal{O})$ injective, since 
$\mathrm{codim}_{A_v}(A_v \setminus U_v) \geq 2$. Moreover, the (extended) line bundle $\eta^* \mathcal{P}_{|A_v}$
is ample (more precisely, it is $r$ times a prinicipal polarization by \cite{zelaci:twistconf} Theorem 6.4),
which implies that the cup-product $\cup [\eta^* \mathcal{P}_{|A_v}]$ is an isomorphism. Therefore,
the image of $f_m$ in $H^1(A_v \setminus U_v, \mathcal{O})$ is non-zero, and by commutativity of the above diagram, 
we deduce that  $f_m \cup [\eta^* \mathcal{P}] \neq 0$.
\end{proof}

Similarly, we can state the analog of the above proposition for the moduli space $\rkdetsanti{r}$.

\begin{proposition} \label{injcupproductPetalecover}
 The linear maps induced by cup product with the Atiyah class $[\mathcal{P}_r] \in 
H^1(\rkdetsanti{r}, \Omega_{\rkdetsanti{r}}^1)$
$$\begin{tikzcd}\cup [\mathcal{P}_r] : H^0(\rkdetsanti{r}, \mathrm{Sym}^m T_{\rkdetsanti{r}}) \ar[r]&  
 H^1(\rkdetsanti{r}, \mathrm{Sym}^{m-1} T_{\rkdetsanti{r}}) \end{tikzcd}$$
 are injective $\forall  m \geq 1$.   
\end{proposition}

\begin{proof}
Same as proof of Proposition \ref{dirim}.
\end{proof}

\begin{remark}
In this paper we will only need the cases $m=1$ and $m=2$, which appear as conditions in 
Theorems \ref{vgdj} and \ref{flatness}.
\end{remark}

\begin{proposition} \label{vanishingH1}
    With the above notation we have  
    $$ H^1(\mathcal{N}, \mathcal{O}) = 0.$$
\end{proposition}

\begin{proof}    
The argument follows very closely the proof given in \cite[Proposition 5.2]{hitchin:1990} in the case of the moduli space
of stable vector bundles with trivial determinant. We outline the argument.  We denote by $W_0^{nod} \subset W_0$ the open subset 
corresponding to integral spectral curves $\pi : X_v \rightarrow C$ having at most one node. 
Then the closed subset $W_0 \setminus W_0^{nod}$ has codimension $2$. Also, since all the fibers of the Hitchin map
$\overline{h}_0 : T^{\vee} \mathcal{N} \rightarrow W_0$ are isotropic for the canonical symplectic form on 
$T^{\vee} \mathcal{N}$ \cite[Theorem 4.8]{zelaci:2022}\footnote{Zelaci shows the statement for the nilpotent cone, but his argument equally works for any fiber.}, hence of dimension $\leq \dim W_0 = \dim \mathcal{N}$,
we see that the preimage $\overline{h}_0^{-1}(W_0^{nod}) = T^{\vee} \mathcal{N}^{nod} \subset T^{\vee} \mathcal{N}$ also has a complement of codimension $2$.
Therefore, by Hartogs's theorem, restriction of classes in $H^1$ gives an inclusion
$$ H^1(\mathcal{N}, \mathcal{O}) \subset H^1(T^{\vee} \mathcal{N},\mathcal{O}) \hookrightarrow  H^1(T^{\vee} \mathcal{N}^{nod},\mathcal{O}). $$
Since $\mathcal{A}^{nod}_0$ is smooth by Lemma \ref{Asmooth} and since 
$\mathrm{codim}_{\mathcal{A}^{nod}_0}(\mathcal{A}^{nod}_0 \setminus T^{\vee} \mathcal{N}^{nod}) \geq 3$ by Corollary \ref{codimensionA},
we can apply Hartogs's theorem to classes in $H^1$ (see e.g. \cite{SGA2:1968} Exposé I Corollaire 2.14 and Exposé VII Corollaire 1.4) and we obtain an isomorphism
$$ H^1(T^{\vee} \mathcal{N}^{nod}, \mathcal{O}) = H^1(\mathcal{A}^{nod}_0, \mathcal{O}). $$

Next, we observe that $h_* \mathcal{O}_{\mathcal{A}^{nod}_0}  = \mathcal{O}_{W_0^{nod}}$ by Zariski's main theorem and 
\begin{equation} \label{firstdirectimageh}
R^1 h_* \mathcal{O}_{\mathcal{A}^{nod}_0} = \mathcal{O}_{W_0^{nod}} \otimes W_0^*. 
\end{equation}
The last equality is seen as follows: as already noticed in \eqref{fibercompactifiedPrym}, for any $v \in W_0^{nod}$ the fiber 
$(h_0^\sigma)^{-1}(v) = \widehat{A}^\sigma_v$ equals the fixed-point locus of $\sigma$ in the fiber $h_0^{-1}(v) = \widehat{A}_v$, which is a closed subvariety of the compactified Jacobian $\widehat{\mathrm{Jac}}(\widetilde{X}_v)$. We again adapt Hitchin's original proof \cite[page 378]{hitchin:1990} to the anti-invariant case. Consider the unique extension of the line bundle 
$\eta^*\mathcal{L}$ to the open subset $h^{-1}_0(B^{nod}_0) \subset \mathrm{Higgs}_{\Ctilda}(r)_0$. It is shown in loc.cit. that the map given by cup-product with the Atiyah class of the restriction $\eta^*\mathcal{L}_{|\widehat{A}_v}$ is an isomorphism
$$
\begin{tikzcd}
\cup \left[\eta^*\mathcal{L}_{|\widehat{A}_v}\right] : H^0(\widehat{A}_v, T)  \ar[r, "{\cong}"] & H^1(\widehat{A}_v, \mathcal{O}). 
\end{tikzcd}
$$
We note that $\sigma$ preserves the line bundle $\eta^*\mathcal{L}_{|\widehat{A}_v}$ and that this cup-product map is
$\sigma$-equivariant. Using Proposition \ref{H0andH1compactifiedPrym} we obtain an isomorphism by restriction to the
$\sigma$-invariant subspace
$$
\begin{tikzcd}
 H^0(\widehat{A}_v^\sigma, T) =  H^0(\widehat{A}_v, T)_+  \ar[r, "{\cong}"] & 
 H^1(\widehat{A}_v, \mathcal{O})_+ =  H^1(\widehat{A}_v^\sigma, \mathcal{O}).
\end{tikzcd}
$$
We then deduce the equality \eqref{firstdirectimageh} as in \cite[page 378]{hitchin:1990}.
\medskip

Thus, by Leray's spectral sequence and the fact that $H^1(W_0^{nod},\mathcal{O}) = 0$, we deduce that
$$ H^1(T^{\vee} \mathcal{N}^{nod},\mathcal{O}) = H^0(W_0^{nod}, \mathcal{O}) \otimes W_0^* = H^0(W_0, \mathcal{O}) \otimes W_0^*,$$
where the last equality is obtained again thanks to Hartogs's theorem.

We can now conclude, similarly as in \cite{hitchin:1990} page 378, by saying that any non-zero class in $H^1(\mathcal{N}, \mathcal{O})$ would
define a non-zero class in $H^0(W_0, \mathcal{O}) \otimes W^*_0$ of homogeneity $-1$, but there are no such classes.
\end{proof}

Using Proposition \ref{cupproductPinjective} with $m=1$ we deduce

\begin{corollary} \label{vanishingH0}
      With the above notation we have  
    $$ H^0(\mathcal{N}, T_{\mathcal{N}}) = 0.$$
\end{corollary}

\subsection{The Prym--Hitchin system for a family of double covers}

We can consider for a family of double covers $p : \calCt \rightarrow \calC$ over $S$ the relative moduli space 
$$\pi : \mathcal{N}_{\calCt/S} \rightarrow S$$ such that for any $s \in S$ we have
$\pi^{-1}(s) = \mathcal{N}$, where $\mathcal{N}$ denotes the moduli space \eqref{moduliN} associated to the cover
$\Ctilda = \calCt_s \rightarrow C = \calC_s$.  For simplicity we will also denote by $\mathcal{N}$
the relative moduli space $\mathcal{N}_{\calCt/S}$. The next proposition is a direct consequence of the base change theorems, Proposition \ref{vanishingH1} and Corollary \ref{vanishingH0}.

\begin{proposition} \label{novect}
Using the above notation, we have for the relative moduli space 
 $\pi : \cN = \cN_{\calCt/S} \rightarrow S$
\begin{enumerate}
    \item $\pi_* T_{\cN} = 0$,
    \item $R^1\pi_* \mathcal{O}_{\cN} = 0$.
\end{enumerate}
\end{proposition}
Similarly, we can also consider as in \eqref{diag_N} the family
$$\widetilde{\pi}_e :  \rkdetsanti{r} \rightarrow S.$$
As was explained in \S 3.1, there is a natural map $\iota_{\mathrm{SL}_r} : \rkdetsanti{r} \rightarrow \cN$ over $S$,
which is an isomorphism if $r$ is odd, and a double étale cover if $r$ is even. We also have

\begin{proposition} \label{dirim}
Using the above notation, we have for the relative moduli space 
 $\widetilde{\pi}_e :  \rkdetsanti{r} \rightarrow S$
\begin{enumerate}
    \item $\widetilde{\pi}_{e*} T_{\rkdetsanti{r}} = 0$,
    \item $R^1\widetilde{\pi}_{e*} \mathcal{O}_{\rkdetsanti{r}} = 0$.
\end{enumerate}
\end{proposition}

\begin{proof}
If $r$ is odd, $\iota_{\mathrm{SL}_r}$ is an isomorphism and we are done. If $r$ is even, we claim that
Proposition \ref{vanishingH1} and Corollary \ref{vanishingH0} also hold for the double étale cover 
$\rkdetsanti{r}$. Without going into the details, we observe that the cotangent bundle $T^{\vee} \rkdetsanti{r}$
is the fiber product of $T^{\vee} \cN$ by $\rkdetsanti{r}$, so it is a double étale cover over $T^{\vee} \cN$. This cover extends
uniquely to $\mathcal{A}_0^{sm}$ and $\mathcal{A}_0^{nod}$. The reader can easily check that the proofs of Proposition \ref{vanishingH1} and Corollary \ref{vanishingH0} remain valid when considering these étale covers.
\end{proof}

\begin{remark}
    For stack aficionados, we can give a shorter proof of the analog of Proposition \ref{vanishingH1} showing that 
    $H^1( \rkdetsanti{r}, \cO) = 0$. Consider (see \S 3.1) the moduli stack $\mathfrak{N}^{\pm}_{\mathrm{SL}_r}$ parameterizing  pairs $(E, \psi)$ of anti-invariant vector bundles $E$ together with an isomorphism
    $\psi : E \rightarrow \sigma^* E^*$. By \cite[Proposition 1]{heinloth:2010} we know that the moduli stack
    $\mathfrak{N}^{\pm}_{\mathrm{SL}_r}$ is smooth and by \cite[Theorem 3]{heinloth:2010} that 
    $\mathrm{Pic}(\mathfrak{N}^{\pm}_{\mathrm{SL}_r}) = \mathbb{Z}$. We then consider the open substack
    $\mathfrak{N}^{\pm,s}_{\mathrm{SL}_r}$ parameterizing pairs $(E, \psi)$ such that $E$ is a stable vector bundle.
    By general results on stacks \cite[Lemma 7.3]{biswashoffmann}, the homomorphism given by restriction of line bundles
    $$ \mathbb{Z} = \mathrm{Pic}(\mathfrak{N}^{\pm}_{\mathrm{SL}_r}) \longrightarrow 
         \mathrm{Pic}(\mathfrak{N}^{\pm,s}_{\mathrm{SL}_r})$$
    is either an isomorphism if $\mathrm{codim}(Z) \geq 2$, or a surjection if $\mathrm{codim}(Z) = 1$, where
    $Z$ denotes the closed substack 
    $\mathfrak{N}^{\pm}_{\mathrm{SL}_r} \setminus \mathfrak{N}^{\pm,s}_{\mathrm{SL}_r}$. In both cases we can 
    conclude that  $\mathrm{Pic}(\mathfrak{N}^{\pm,s}_{\mathrm{SL}_r})$ is discrete. Moreover, the surjective
    classifying map $\mathfrak{N}^{\pm,s}_{\mathrm{SL}_r} \rightarrow \rkdetsanti{r}$ to the coarse moduli space
    induces an injective homomorphism $\mathrm{Pic}(\rkdetsanti{r}) \hookrightarrow \mathrm{Pic}(\mathfrak{N}^{\pm,s}_{\mathrm{SL}_r})$, from which we conclude that $\mathrm{Pic}(\rkdetsanti{r})$ is also discrete. Hence the trivial bundle $\mathcal{O}$ on  $\rkdetsanti{r}$ has no infinitesimal deformations, which implies that $H^1( \rkdetsanti{r}, \cO) = 0$.
\end{remark}

\section{The Prym--Hitchin connection}\label{atlast}

Still in the setting of the previous section, we now give the construction of a flat projective connection on the locus of anti-invariant vector bundles. According to Theorem \ref{vgdj}, the existence of this connection follows from the construction of a symbol map verifying conditions \ref{vgdj-one} through \ref{vgdj-three}, whereas in Theorem \ref{flatness} we will use Zelaci's Prym--Hitchin integrable system in lieu of the original one.

\subsection{}
Using the morphism $\iota_{\pm}$ from an open subvariety of $\rkdetsanti{r}$
to a smooth subvariety of $\rkdetstilde{r}$, we have the canonical splitting of the pull-back of the ambient tangent bundle, 
as in (\ref{split}).
We will denote the canonical projection $\iota^*_{\pm}T_{\rkdetstilde{r}/S}
\rightarrow T_{\rkdetsanti{r}/S}$ (which averages a tangent vector and its image under the involution $\sigma$) by $p_\sigma$.

Both the construction of the symbol map and the verification of the conditions in Theorem \ref{vgdj} for $\rkdetsanti{r} \rightarrow S$ now proceed through use of the splitting of the ambient tangent and cotangent bundles along $\rkdetsanti{r}$.
\begin{definition}\label{hitchinsymbol}
  The \emph{Prym--Hitchin symbol} $\rho^{\PH}$ is the composition
  \[
    \begin{tikzcd}[column sep=5em]
      R^1 \pit_{s \ast} T_{\calCt/S} \arrow[r, "\rho^{\Hit}"] \arrow[rr, bend right=10, swap, "\rho^{\PH}"] &
      \pit_{e \ast} \Sym^2 T_{\rkdetstilde{r}/S} \arrow[r, "\pit_{e \ast}(\Sym^2 p_\sigma \circ \iota_{\pm}^\ast)"] &
      \pit_{e \ast} \Sym^2 T_{\rkdetsanti{r}/S}.
    \end{tikzcd} 
  \]
\end{definition}

\begin{theorem}\label{PHconnection}

  The  direct image sheaf $\pit_{e \ast}\left(\calP_r^{\otimes k}\right)$ 
   
     carries a unique projective connection with symbol map
  \[
    \rho = \frac{2}{r+k} (\rho^{\PH} \circ \kappa_{\calCt/S}) .
  \]  
\end{theorem}
\begin{proof}
  To verify condition \ref{vgdj-one} of Theorem \ref{vgdj}, we need to compute $\mu_{\calP^{\otimes k}} \circ \rho$; note that we have a diagram
  \[
    \begin{tikzcd}[column sep=5em]
      \pit_{e \ast} \Sym^2 T_{\rkdetstilde{r}/S} \arrow[r, "\pit_{e \ast}(\Sym^2 p_\sigma \circ \iota_{\pm}^\ast)"] \arrow[d, "{\cup[\cL_r]_{\rkdetstilde{r}/S}}"] & \pit_{e \ast} \Sym^2 T_{\rkdetsanti{r}/S} \arrow[d, "{\cup[\iota_{\pm}^\ast \cL_r]_{\rkdetsanti{r}/S}}"] \\
      R^1 \pit_{e \ast} T_{\rkdetstilde{r}/S} \arrow[r, "R^1 \pit_{e \ast}(p_\sigma \circ \iota_{\pm}^\ast)"] & R^1 \pit_{e \ast} T_{\rkdetsanti{r}/S}
    \end{tikzcd}
  \]
  which is commutative since $\sigma^\ast \cL_r \cong \cL_r$, and therefore the relative first Chern class $[\cL_r]_{\rkdetstilde{r}/S}$ 
  restricted to the anti-invariant locus has no components along the conormal bundle. It is immediate to check that the Kodaira--Spencer map of 
  $\rkdetsanti{r}/S$ factors as
  \[
    \kappa_{\rkdetsanti{r}/S} = R^1 \pit_{e \ast}(p_\sigma \circ \iota_{\pm}^\ast) \circ \kappa_{\rkdetstilde{r}/S} .
  \]
  Condition \ref{vgdj-one} is now verified by commutativity of the big triangle in the following diagram, which follows from that of the small triangle (by the original case of Hitchin's construction Theorem \ref{existenceconnection}) and of the square which we just checked:
  \[
    \begin{tikzcd}[row sep = small, column sep = 5em]
     & & & & \pit_{e \ast} \Sym^2 T_{\rkdetsanti{r}/S} \arrow[dddd, "{\begin{array}{c} 2 \mu_{\calP_r^{k}} = \\ 
      \cup 2(r+k)[\calP_r] = \\ \cup (r+k)[\iota_{\pm}^\ast \cL_r] \end{array}}" '] \\
    &  & \pit_{e \ast} \Sym^2 T_{\rkdetstilde{r}/S} \arrow[dd, "{\begin{array}{c} \mu_{\cL^{k}_r}= \\ \cup (r+k)[\cL_r] \end{array}}"] \arrow[rru, swap, sloped, pos=0.5, "\pit_{e \ast}(\Sym^2 p_\sigma \circ \iota_{\pm}^\ast)"] & & \\
      T_S \arrow[rru, sloped, swap, pos=0.7, "\frac{1}{r+k}\rho^{\Hit}\circ \kappa_{\calCt/S}"] \arrow[rrd, sloped, swap, pos=0.7, "-\kappa_{\rkdetstilde{r}/S}"] \arrow[rrrrdd, bend right=20, swap, "-\kappa_{\rkdetsanti{r}/S}"] \arrow[rrrruu, bend left=20, "\frac{1}{r+k} \rho^{\PH} \circ \kappa_{\calCt/S}"] & & & & \\
   &   & R^1 \pit_{e \ast} T_{\rkdetstilde{r}/S} \arrow[rrd, sloped, pos=0.5,  "R^1 \pit_{e \ast} (p_\sigma \circ \iota_{\pm}^\ast)"] & & \\
     & & & & R^1 \pit_{e \ast} T_{\rkdetsanti{r}/S}.
    \end{tikzcd}
  \]
  Note that we used the fact that $\mu_{\calP_r^{k}} = \cup (r+k)[\calP_r]$, which follows from Remark \ref{exprcanpf} and from
  \cite[Proposition 3.6.1]{BBMP:2020}.
  
  As for condition \ref{vgdj-two}, we know from Proposition \ref{dirim}(2) that $R^1 \pit_{e \ast} \cO_{\rkdetsanti{r}} = 0$, so that surjectivity is trivially satisfied. 

  Finally, condition \ref{vgdj-three} is a consequence of Corollary \ref{H0N}.
\end{proof}

\subsection{} As already stated above, the flatness of the Prym--Hitchin connection follows now from an analogous reasoning as before:

\begin{lemma}\label{componentsymb}
The second-order symbols of the projective heat operators which define the connections of Theorem \ref{PHconnection}, seen as functions on the cotangent bundle $T^{\vee}_{\rkdetsanti{r}/S}$, are components of the Prym--Hitchin integrable system.
\end{lemma}
\begin{proof}
  We recall from \cite[Proposition 4.4]{zelaci:2022} that the quadratic part of the 
  Prym--Hitchin integrable system is given by the composition
  \[ \begin{tikzcd}
  T^{\vee}_{\rkdetsanti{r}/S} \ar[r,"j"]& \iota_{\pm}^{\ast}T^{\vee}_{\rkdetstilde{r}/S}
  \ar[r] & \pit_{s \ast} K_{\calCt/S}^{ 2}
  \ar[r,"p_+"]& (\pit_{s \ast} K_{\calCt/S}^{ 2})_+ \cong \pi_{s \ast} K_{\calC/S}^{2},\end{tikzcd}
  \]
  where the middle arrow is the quadratic part of the original Hitchin system. As in the discussion preceding Definition \ref{hitchinsymbol}, factoring through symmetric squares and dualizing, in order to show the claim we need to verify that
  \[
    \Sym^2 j^\ast = \Sym^2 (p_\sigma \circ \iota_{\pm}) ,
    \qquad \textrm{ and } \qquad
    p_+^\ast \circ \kappa_{\calC/S} = \kappa_{\calCt/S} .
  \]
  The first of these conditions is again clear in view of the splitting of the restriction of the relative tangent and cotangent bundle. The second follows since we can calculate the Kodaira--Spencer map of the family $\calCt \to S$ in the \v{C}ech formalism from that of $\calC \to S$ by simply lifting vector fields on $\calC$ to $\sigma$-invariant vector fields on $\calCt$.
\end{proof}

\begin{theorem}\label{PHFlat} The Prym--Hitchin connection is projectively flat.
\end{theorem}
\begin{proof}  Lemma \ref{componentsymb} verifies condition (\ref{flatness-one}) of Theorem \ref{flatness}. Proposition
\ref{injcupproductPetalecover} for $m=2$ verifies condition (\ref{flatness-two}) since $\mu_{\mathcal{P}_r^k} = 
\cup (r+k)[\mathcal{P}_r]$. Proposition \ref{dirim}(1) verifies condition (\ref{flatness-three}). 
\end{proof}

 \section{A Laszlo theorem for anti-invariant bundles}\label{twistedlaszlo}
\subsection{} A key ingredient in the standard theory of non-abelian theta function is the correspondence with spaces of conformal blocks \cite{beauville.laszlo:1994}.  These form natural vector bundles on the moduli spaces of pointed projective curves $\mathcal{M}_{g,n}$ (and in fact also their Deligne-Mumford compactification $\overline{\mathcal{M}}_{g,n}$ obtained by adding stable curves, though we will not make use of this).  There exists a twisted $\mathcal{D}$-module structure (or, in the terminology of \cite{looijenga:2013}, a $\lambda$-flat connection), known as the WZW or TUY connection, on these as well \cite{TUY:1989, tsuchimoto:1993}. Each  twisted $\mathcal{D}$-module induces a flat projective connection, and it was shown by Laszlo that the natural isomorphism between two projective bundles of non-abelian theta functions and conformal blocks (or rather the descent of the latter to $\mathcal{M}_g$) is flat with respect to the Hitchin and WZW connections (\cite{laszlo:1998}, see also \cite[\S 5.6]{ueno:2008}).  Another approach to this, based on localisation of vertex algebras, was given by Ben-Zvi and Frenkel in \cite{ben-zvi.frenkel:2004}.  More recently, Laszlo's Theorem was also extended to the case of parabolic bundles in \cite{biswas2023geometrization}.

The counterpart to non-abelian theta functions on moduli spaces of torsors for parahoric Bruhat-Tits group schemes is given by the spaces of so-called \emph{twisted} conformal blocks.  In this set-up one considers a finite group $\Gamma$ which acts on a curve $\calCt$, as well as on a semi-simple Lie algebra $\mathfrak{g}$.  In the case where $\calCt$ is smooth (and varies in a family parametrised by $S$), a connection on the corresponding vector bundle was constructed by Szcesny \cite{szcesny:2006}, in the framework of vertex algebras.  Various generalizations of this were since given by Damiolini \cite{damiolini:2020}, Hong and Kumar \cite{hong.kumar:2018}, and Deshpande and Mukhopadhyay \cite{deshpande.mukhopadhyay:2023}.

Very relevant for us, Hong and Kumar also established the isomorphism between non-abelian theta functions and twisted conformal blocks \cite[Theorem 12.1]{hong.kumar:2018}.  In the case where $\mathfrak{g}=\mathfrak{sl}_r$, $\Gamma=\mathbb{Z}/2\mathbb{Z}$ and $\calCt\rightarrow \calC$ is ramified, this was also shown by Zelaci \cite[Theorem 5.4]{zelaci:twistconf}.

In this section we will show that, in the case where $\mathfrak{g}=\mathfrak{sl}_r$, $\Gamma=\mathbb{Z}/2\mathbb{Z}$ and $\calCt$ is as in Section~\ref{higherpryms}, the isomorphism of Hong and Kumar respects the twisted WZW connection and the Prym-Hitchin connection.  This is of interest in its own right, but will also help us in Section \ref{generalflat} to use conformal embeddings (see Appendix \ref{appendixtwistedconf}), which give rise to flat maps between bundles of twisted conformal blocks (note that even in the non-twisted case the corresponding flatness property for maps between bundles of non-abelian theta functions is proven in \cite{belkale:2009} using a necessary excursion to conformal blocks).

Our strategy will broadly follow \cite[Section 8]{laszlo:1998}, and more particularly the approach to Laszlo's theorem given in \cite[\S 5.6]{ueno:2008} (Laszlo's original proof characterised the Hitchin connection infinitesimally, as was done in \cite{hitchin:1990}, rather than locally via Theorem \ref{vgdj}), see also \cite{biswas2023geometrization} for the equivalent story in the parabolic case.  Essentially, we will show that the construction of the twisted WZW connection can be recast to construct a projective heat operator with the same symbol that determines the Prym-Hitchin connection.  We will follow Damiolini's construction of the relevant twisted WZW connection \cite[\S 4]{damiolini:2020}, a summary of which is given in Appendix \ref{twistedWZWoverview}.

\subsection{The WZW connection as a heat operator} 
Following the notation from Appendix \ref{appendixtwistedconf},  we will denote the parahoric Bruhat-Tits group schemes on $X$ obtained by invariant Weil restriction by $\mathcal{SL}_r^{\pm}$ (using the involutions (\ref{syminvol}) and (\ref{altinvol}) respectively)
and the stack of $\mathcal{SL}_r^{\pm}$-torsors on $X$ as $\mathcal{B}un_{\mathcal{SL}_r^{\pm}}$. As was discussed in Section \ref{BTtorsor}, we have an equivalence of stacks $\mathcal{B}un_{\mathcal{SL}_r^{\pm}}\cong \mathfrak{N}_{SL_r}^{\pm}$.  The stable loci in these stacks are denoted as $\mathcal{B}un_{\mathcal{SL}_r^{\pm}}^s$ and $\mathfrak{N}^{\pm,s}_{SL_r}$. By Heinloth's uniformization theorem \cite{heinloth:2010} (conjectured by Pappas-Rapoport \cite{pappas.rapoport:2010}),  this can be expressed as a quotient $$\mathcal{B}un_{\mathcal{SL}_r^{\pm}}\cong \mathfrak{Q}_{\mathcal{SL}_r^{\pm}}\Big/ \mathcal{SL}_r^{\pm}(\calC^{\degree}),$$
where $\mathfrak{Q}_{\mathcal{SL}_r^{\pm}}$ is the twisted affine Grassmannian and $\calC^{\degree}$ is the punctured curve.  We will denote by $\mathfrak{Q}^s_{\mathcal{SL}_r^{\pm}}$ the locus that descends to $\mathcal{B}un_{\mathcal{SL}_r^{\pm}}^s$.  By the affine version of the Borel-Weil theorem \cite{mathieu:1986,kumar:1987}, we have that $$\mathcal{H}^{\mathcal{SL}_r^{\pm}}_{\ell}\cong \pi^{\mathfrak{Q}_{\mathcal{SL}_r^{\pm}}}_*(L^{\otimes \ell})=\pi^{\mathfrak{Q}^s_{\mathcal{SL}_r^{\pm}}}_*(L^{\otimes \ell}),$$ where $\mathcal{H}^{\mathcal{SL}_r^{\pm}}_{\ell}$ is the irreducible quotient of the Verma module as in \ref{twistconfoverview}(\ref{vermquot}), and $L$ is the positive generator of the Picard group of $\mathfrak{Q}_{\mathcal{SL}_r^{\pm}}$.
Similar to \cite[5.2.12]{beilinson-drinfeld:1991} or \cite[(8.6)]{laszlo:1998}, the algebra action of $\overline{U}\widehat{\mathfrak{h}}^{\mathcal{SL}_r^{\pm}}_{\mathbb{D}^{\degree}}$ on $\mathcal{H}^{\mathcal{SL}_r^{\pm}}_{\ell}$ corresponds to a morphism $$\begin{tikzcd}\left(\overline{U}\widehat{\mathfrak{h}}^{\mathcal{SL}_r^{\pm}}_{\mathbb{D}^{\degree}}\right)^{\operatorname{opp}}\ar[r]&\pi_{*}^{\mathfrak{Q}_{\mathcal{SL}_r^{\pm}}}\left(\mathcal{D}_{\mathfrak{Q}/S}(L)\right),\end{tikzcd}$$ where $\operatorname{opp}$ denotes the opposite algebra, and $\mathcal{D}_{\mathfrak{Q}/S}(L)$ refers to the sheaf of relative differential operators on $L$. (This can be understood as an instance of Beilinson-Bernstein localization.)  This morphism now descends  to a morphism
$$\begin{tikzcd}\left(\overline{U}\widehat{\mathfrak{h}}^{\mathcal{SL}_r}_{\mathbb{D}^{\degree}}\right)^{\operatorname{opp}}\ar[r]& \pi_{*}^{\rkdetsanti{r}}\left(\mathcal{D}_{\rkdetsanti{r}/S}(\LL_r)\right).\end{tikzcd}$$ (As in \cite[(8.6)]{laszlo:1998}, this is done, strictly speaking, using a quasi-section of \break $\mathcal{B}un_{\mathcal{SL}_r^{\pm}}^s\rightarrow \rkdetsanti{r}$ defined using an \'etale surjective morphism to $\cN_{\SL_r}^{\pm,s}$.  Such a quasi-section is guaranteed to exist by \cite{heinloth:2010}.) Together with steps (\ref{twistedsegsug}), (\ref{segsugdescent}), (\ref{segsugproj}), and (\ref{finalquot}) from Appendix \ref{twistedWZWoverview} (using the fact that the Segal-Sugawara construction is quadratic), we get a morphism
$$\begin{tikzcd}D^{\scaleto{WZW}{2.5pt}}:T_S\ar[r] & \mathcal{W}_{\rkdetsanti{r}/S}(\LL_r)\Big/\mathcal{O}_S.\end{tikzcd}$$
We can summarise this discussion as
\begin{proposition} 
The map $D^{\scaleto{WZW}{2.5pt}}$ defines a projective heat operator on $L$.
\end{proposition}
We now have
\begin{proposition}\label{symb-twisted-conf}
The symbol of the heat operator $D^{\scaleto{WZW}{2.5pt}}$ is given by $\rho$ from Theorem \ref{PHconnection}.
\end{proposition}
\begin{proof} The proof will follow a similar approach as \cite[\S 8.10-8.14]{laszlo:1998} (see also \cite[p. 128]{ueno:2008}, \cite[\S 6.3-6.4]{bmw:2024}): we will calculate the symbol of the heat operator associated with the twisted WZW-connection (as outlined above and in Appendix \ref{appendixtwistedconf}), and then express the symbol of the Prym-Hitchin heat operator in similar terms, so that they can be compared.  

The former is expressed using the curve $\calC$, and the calculation uses the explicit expression for the Segal-Sugawara construction.  The latter is defined using the double cover $\calCt$, and in order to compare its symbol we calculate it in \v{C}ech cohomology using the formal cover given by the twice punctured curve $\calCt^{\degree\degree}$, and two formal neighbourhoods of the punctures that are interchanged by the involution.

We begin by looking at the twisted WZW connection.  Following \cite[Remark 4.15]{damiolini:2020}, we can express the Sugawara operators associated to $D_j=-z^{j+1}\frac{d}{dz}$ (for a formal coordinate $z$ on $\mathbb{D}$) concretely.  We have\footnote{Note that there is a difference with the signs in the corresponding definitions in \cite[Corollary 3.2]{looijenga:2013} and \cite[Page 1667]{damiolini:2020}, related to a sign difference in the generators $D_j$ of the Virasoro algebras in these papers.  Our sign agrees with the more standard conventions for the Virasoro algebra and Sugawara operators, as in e.g. \cite{kac:1990,kac.raina:2013}.}
$$T_{\mathfrak{h}}(D_j)=\frac{1}{2\left(c_{\mathfrak{sl}_r}+h^{\vee}_{\mathfrak{sl}_r}\right)}\sum_i \circdot D_j(A_i)\circ B_i \circdot,$$
where the $A_i$ and $B_j$ are dual bases for $\omega_{\mathbb{D}^{\degree}}\otimes \mathfrak{h}^{\mathcal{SL}_r^{\pm}}_{\mathbb{D}^{\degree}}$ and $\mathfrak{h}^{\mathcal{SL}_r^{\pm}}_{\mathbb{D}^{\degree}}$ with respect to the perfect pairing $\langle.|.\rangle=\operatorname{res}_{\degree}^{\mathbb{D}}{(.|.)}$, and $\circdot .\ .\circdot$ refers to the normal ordering.  
If $C_k$ is an orthonormal basis of $\mathfrak{sl}_r$ with respect to the normalized Killing form, we can put \begin{equation}\label{twistedbasis}\widetilde{C}_k=\frac{1}{2}\left(C_k+d\Psi^{\pm} (C_k)\right),\end{equation} where $\Psi^{\pm}$ are as in (\ref{syminvol}) and (\ref{altinvol}). (We are thinking here of the $C_k$ and $d\Psi^{\pm} (C_k)$ as belonging to the two different marked points on $\calCt$, and are implicitly assuming these can be globally distinguished, which is always the case \'etale-locally.) We can then let the index $i$ for the $A_i$ and $B_i$ run over pairs $(k,l)\in \{1,\dots, \dim(\mathfrak{sl}_r)\}\times \mathbb{Z}$, and take $$A_i=\widetilde{C}_k z^{-l-1}dz \hspace{.5cm}\text{and}\hspace{.5cm} B_i=\widetilde{C}_kz^l.$$
To calculate the (vertical) symbol of the associated differential operator, we need to evaluate (for any $j\in\mathbb{Z}$, and any vector field $\tau_j$ on $S$ whose image under the Kodaira-Spencer map is represented by $D_j$) $$\langle \phi\otimes \phi|\rho_{D^{\scaleto{WZW}{2.5pt}}}(\tau_j)\rangle=\langle \overline{\phi}\otimes \overline{\phi}|T_{\mathfrak{h}}(D_j)\rangle$$
for any $\phi \in \pi_*T^{\vee}\rkdetsanti{r}$ mapping to $\overline{\phi}\in \mathfrak{h}^{\mathcal{SL}_r^{\pm}}\otimes \Omega_{\mathbb{D}^{\degree}}$.  It suffices to look at $\overline{\phi}$ of the form $\widetilde{C}_mz^ndz$.  We get (unpacking the normal ordering)
\begin{equation}\label{symbwzwcalc}
\begin{split}
2\left(c_{\mathfrak{sl}_r}+h^{\vee}_{\mathfrak{sl}_r}\right)&\left\langle \overline{\phi}\otimes \overline{\phi}|T_{\mathfrak{h}}(D_j)\right\rangle \\
=& 
\left\langle \widetilde{C}_mz^ndz\otimes \widetilde{C}_mz^ndz\big|\sum_i \circdot D_j(A_i)\circ B_i \circdot \right \rangle \\ 
=&\ \sum_{k}\sum_{-1+j<2l}\operatorname{res}_{\degree}^{\mathbb{D}}\left(\widetilde{C}_mz^ndz |-\widetilde{C}_kz^{-l-1+j +1}\right)\operatorname{res}_{\degree}^{\mathbb{D}}\left(\widetilde{C}_mz^ndz |\widetilde{C}_kz^l\right)\\
&+\sum_{k}\sum_{-1+j=2l}\frac{1}{2} \operatorname{res}_{\degree}^{\mathbb{D}}\left(\widetilde{C}_mz^ndz |-\widetilde{C}_kz^{-l-1+j+1}\right)\operatorname{res}_{\degree}^{\mathbb{D}}\left(\widetilde{C}_mz^ndz |\widetilde{C}_kz^l\right)\\
&+\sum_{k}\sum_{-1+j=2l}\frac{1}{2}\operatorname{res}_{\degree}^{\mathbb{D}} \left(\widetilde{C}_mz^ndz |\widetilde{C}_kz^l\right)\operatorname{res}_{\degree}^{\mathbb{D}}\left(\widetilde{C}_mz^ndz |-\widetilde{C}_kz^{-l-1+j+1}\right)\\
&+ \sum_{k}\sum_{-1+j>2l}\operatorname{res}_{\degree}^{\mathbb{D}}\left(\widetilde{C}_mz^ndz |\widetilde{C}_kz^l\right)\operatorname{res}_{\degree}^{\mathbb{D}}\left(\widetilde{C}_mz^ndz |-\widetilde{C}_kz^{-l-1+j+1}\right)\\
=& \sum_{l}-\delta_{n+j,l-1}\delta_{n+l,-1}\\
=&\ -\delta_{2n+j,-2}.
\end{split}\end{equation}

We can now turn our attention to the Prym-Hitchin connection.  To be able to compare the symbol of the Prym-Hitchin projective heat operator to that of the twisted WZW one, we have to express it in comparable terms.  For the Hitchin heat operator this is done in \cite{laszlo:1998} by expressing the symbol, in \v{C}ech terms, using the formal open covering of $\calC$ given by the punctured curve $\calC^{\degree}$ and the formal disk $\mathbb{D}$.  

The Prym-Hitchin symbol was defined in Definition \ref{hitchinsymbol} and Theorem \ref{PHconnection} by projecting down from the Hitchin symbol for $\calCt$.  In order to do this in the \v{C}ech language, we will first express the Hitchin symbol for $\calCt$ using a formal cover made of the twice punctured curve $\calCt^{\degree\degree}$, and two formal open disks $\widetilde{\mathbb{D}}=\mathbb{D}\coprod \mathbb{D}$, interchanged under the involution $\sigma$ (\'etale locally these can always be discriminated, which we will do for convenience).  The average of the two residues at the marked points gives an isomorphism
$$\operatorname{res}^{\widetilde{{\mathbb{D}}}}_{\degree \degree}:R^1\pi_* K_{\calCt/S}\cong \mathcal{O}_S$$ (the factor $\frac{1}{2}$ is there to make the isomorphism commensurate with the isomorphism $\operatorname{res}^{\mathbb{D}}_{\degree}$ on $\calC$).  The fact that we are projecting down the ordinary Hitchin symbol for $\calCt$ means that we are only interested in families of $\calCt$ coming from families $\calCt\rightarrow \calC$ of \'etale double covers, hence the elements of $R^1\pi T_{\calCt/S}$ we will see can be represented by identical pairs $(t,t)\in T^{\oplus 2}_{\mathbb{D}^{\degree}/S}$.  

We now want to evaluate the Hitchin symbol $\rho^{\operatorname{Hit}}$ on $\phi\otimes \phi$, where $$\phi\in \pi_* T^{\vee}_{\rkdetsanti{r}/S}\subset \pi_* T^{\vee}_{\rkdetstilde{r}/S}\cong \pi_*(\mathcal{E}nd^{0}(\mathcal{E})\otimes K_{\calCt/S}).$$  As we have an injection given by restriction
$$\begin{tikzcd}\pi_*(\mathcal{E}nd^{0}(\mathcal{E})\otimes K_{\calCt/S})\ar[r, hook]& \mathfrak{sl}_r\otimes \left(\Omega^1_{\mathbb{D}^{\degree}/S}\oplus\Omega^1_{\mathbb{D^{\degree}}/S}\right),\end{tikzcd}$$ 
we can represent such $\phi$ as $(\overline{\phi}_1,\overline{\phi}_2)$ given by pairs in $\mathfrak{sl}_r\otimes \left(\Omega^1_{\mathbb{D}^{\degree}/S}\oplus\Omega^1_{\mathbb{D}^{\degree}/S}\right)$ that are exchanged by the isomorphism $d\Psi^{\pm}\otimes \sigma^* $. 
We have (see \cite[(8.5)]{laszlo:1998})
$$\rho^{\Hit}(t,t)(\phi\otimes\phi)=\operatorname{res}^{\widetilde{{\mathbb{D}}}}_{\degree \degree}\ \Big(\left(\overline{\phi}_1,\overline{\phi}_2\right)\Big|(t,t).\left(\overline{\phi}_1,\overline{\phi}_2\right)\Big),$$ where $(.|.)$ is the normalized Killing form on $\mathfrak{sl}_r$.
Using a local formal coordinate $z$ on $\mathbb{D}$ as before, it suffices to look at  $t\in T_{\mathbb{D}^{\degree}}$ of the form $t=D_j=-z^{j+1}\frac{d}{dz}$, for some $j\in\mathbb{Z}$.  In the same vein, to compare with the symbol we calculated for $D^{\scaleto{WZW}{2.5pt}}$, we can just take $(\overline{\phi}_1,\overline{\phi_2})$ of the form  
$(C_mz^ndz, d\Psi^{\pm}(C_m)z^ndz)$. 
We then get 
\begin{equation}\label{symbphcalc}\begin{split} 
2(k+r)&\rho^{\Hit}(t,t)(\phi\otimes \phi) \\
&= \operatorname{res}_{\degree \degree}^{\widetilde{\mathbb{D}}}\  \Big(\left(\overline{\phi}_1,\overline{\phi}_2\right)\Big|(t,t).\left(\overline{\phi}_1,\overline{\phi}_2\right)\Big)\\ & = \operatorname{res}_{\degree \degree}^{\widetilde{\mathbb{D}}}\  \Big(\left(C_mz^ndz, d\Psi^{\pm}(C_m)z^ndz\right)|\left(D_j C_mz^ndz,D_jd\Psi^{\pm}(C_m)z^ndz\right)\Big)\\
& = \operatorname{res}_{\degree \degree}^{\widetilde{\mathbb{D}}}\left(-z^{2n+j+1}dz,-z^{2n+j+1}dz\right)\\ 
& =  -\delta_{2n+j,-2}.
\end{split}\end{equation}
Comparing (\ref{symbwzwcalc}) and (\ref{symbphcalc}) now gives the result.
\end{proof} 
This allows us to conclude
\begin{theorem}\label{equivlaszlo}
The projective connection induced by $D^{\scaleto{WZW}{2.5pt}}$ (the twisted WZW connection) is the same as the Prym-Hitchin connection.
\end{theorem}
\begin{proof} This follows immediately from Theorem \ref{PHconnection}, Proposition \ref{symb-twisted-conf} and the uniqueness in Theorem \ref{vgdj}.
\end{proof}

\section{Anti-invariant level-rank duality}\label{level1}
\subsection{General statement of anti-invariant level-rank duality}
We can now state the Prym version of the level-rank duality we will be considering. 
We consider a fixed double étale cover $p :\widetilde{C} \rightarrow C$ and we denote by 
$\mathcal{N}^{*,+,ss}_{\GL_r}$ the moduli space 
of semi-stable rank-$r$ vector bundles $E$ on $\widetilde{C}$ that come equipped with an isomorphism 
$$\phi:\begin{tikzcd}\sigma^*(E)\ar[r,"\cong"]& E^{\vee} \otimes K_{\widetilde{C}}\end{tikzcd}
\qquad \text{satisfying} \qquad \phi= \sigma^\ast \phi^t. $$ 
Hence $\deg(E) = r(\widetilde{g}-1)$. If we choose the canonical linearisation on $K_{\widetilde{C}}= p^* K_C$,
then, as explained in \cite[\S 4]{zelaci:2019}, we can equip $p_*(E)$ with a non-degenerate $K_C$-valued quadratic form
and deduce, by a result of Mumford \cite[Page 184]{mumford:1971}, that $h^0(\widetilde{C},E) \  \text{mod} \  2$ is constant on each connected component
of $\mathcal{N}^{*,+,ss}_{\GL_r}$. Hence we can decompose
$$ \mathcal{N}^{*,+,ss}_{\GL_r} = \mathcal{N}^{{\rm even},+,ss}_{\GL_r} \sqcup \mathcal{N}^{{\rm odd},+,ss}_{\GL_r}. $$
Note that the moduli space $\mathcal{N}^{*,-,ss}_{\GL_r}$, parameterizing pairs $(E,\phi)$ as above with \break
$\phi= - \sigma^\ast \phi^t$, is non-canonically isomorphic to $\mathcal{N}^{*,+,ss}_{\GL_r}$.
In the particular case $r=1$, we will denote 
$$ \mathcal{N}^{{\rm even},+,ss}_{\GL_1} = P^{\rm even} \qquad \text{and}  \qquad
\mathcal{N}^{{\rm odd},+,ss}_{\GL_1} = P^{\rm odd}$$
the two connected components of $\mathrm{Nm}^{-1}(K_C) \subset \Pic_{\widetilde{C}}^{\widetilde{g}-1}$. 
We also note that for the $-$components we have
$\mathcal{N}^{*,-,ss}_{\GL_1} = \mathrm{Nm}^{-1}(K_C\otimes \Delta) \subset \Pic_{\widetilde{C}}^{\widetilde{g}-1}$.

Recall that on $\rksstartilde{r}$ (the moduli space of stable, rank-$r$ bundles on $\widetilde{C}$ of degree $r(\widetilde{g}-1)$) there exists a natural divisor $\Theta$, which as a set is given by $$\Theta=\{E\in \rksstartilde{r}\ |\ h^0(\widetilde{C},E)\neq 0\}.$$  The analog for anti-invariant bundles is given by
\begin{proposition}[{\cite{zelaci:twistconf}}]
There is a natural Cartier divisor $\Xi$ on $\mathcal{N}^{{\rm even},+,ss}_{\GL_r}$, defined on the stable locus as the underlying reduced divisor of the restriction of the divisor $\Theta$ on $\rksstartilde{r}$.  We have that $h^0(\mathcal{N}^{{\rm even},+,ss}_{\GL_r},\mathcal{O}(\Xi))=1$.  
\end{proposition}
\begin{proof} This property is essentially Lemmas 6.2 and 6.3 in \cite{zelaci:twistconf}, though it is formulated there slightly differently, in terms of a divisor $\Xi_L$ on one connected component of $\mathcal{N}^{+,s}_{\GL_r}$, depending on a line bundle $L$ in 
$P^{\rm even}$, and  supported on bundles $E$ such that $h^0(E\otimes L)\neq 0$.  Tensoring with this $L$ gives an isomorphism of this connected component of $\mathcal{N}^{+,s}_{\GL_r}$ with  $\mathcal{N}^{{\rm even},+,s}_{\GL_r}$. We then can extend the divisor from $\mathcal{N}^{{\rm even},+,s}_{\GL_r}$
to $\mathcal{N}^{{\rm even},+,ss}_{\GL_r}$ by Hartogs's theorem.
\end{proof}
We will put $\calP_{\Xi,r}=\mathcal{O}(\Xi)$.
Because semi-stability is preserved by tensor product and because of \cite[Proposition 4.2]{zelaci:2019} we have a morphism induced by tensor product
$$\begin{tikzcd}t:\mathcal{N}^{+,ss}_{\SL_r}\times \mathcal{N}^{{\rm even},+,ss}_{\GL_k}\ar[r]& \mathcal{N}^{{\rm even},+,ss}_{\GL_{kr}}. \end{tikzcd}$$

If $r$ is even, we can extend this map to the odd-component of $\mathcal{N}^{*,+,ss}_{\GL_k}$
$$\begin{tikzcd}t:\mathcal{N}^{+,ss}_{\SL_r}\times \mathcal{N}^{{\rm odd},+,ss}_{\GL_k}\ar[r]& \mathcal{N}^{{\rm even},+,ss}_{\GL_{kr}} \end{tikzcd}$$
Indeed, for any vector bundles $E \in \mathcal{N}^{+,ss}_{\SL_r}$ and $F \in \mathcal{N}^{{\rm odd},+,ss}_{\GL_k}$ we see that $h^0(E \otimes F)$
is even, since this number is constant under deformation and is even when $E$ is the trivial anti-invariant bundle $\mathcal{O}^r$.
We have, from the seesaw theorem \cite[Corollary 6, \S 5]{mumford:2008}, that 
$$t^*\calP_{\Xi,kr}\cong \calP^k_r\boxtimes \calP_{\Xi,k}^r.$$  
If $r$ is even, we define the line bundle $\calP_{\Xi,k}^r$ on $\mathcal{N}^{{\rm odd},+,ss}_{\GL_k}$ as the 
restriction\break $\calP_{\Xi,k}^r = \mathcal{O}(\frac{r}{2} \Theta)|_{\mathcal{N}^{{\rm odd},+,ss}_{\GL_k}}$.

Considering the above introduced moduli spaces for the relative curve $p: \widetilde{\mathcal{C}} \rightarrow
\mathcal{C}$ over $S$ and denoting by $\pi$ the projection maps to $S$, we obtain by taking the direct image 
under $\pi$ of these line bundles to $S$ a morphism of locally free sheaves over $S$ \begin{equation}\label{strangedualityatlast}\begin{tikzcd}
SD: \left(\pi_*\calP^r_{\Xi,k}\right)^{\vee} \ar[r] & \pi_* \calP^k_r,\end{tikzcd} \end{equation}
which we refer to as the \emph{anti-invariant level-rank duality} or \emph{anti-invariant strange duality}. Note
that the moduli spaces appearing on the LHS of (\ref{strangedualityatlast}) depend on $r$: if $r$ is odd, we take the component $\mathcal{N}^{{\rm even},+,ss}_{\GL_k}$ and if $r$ is even, we take the two components
$\mathcal{N}^{*,+,ss}_{\GL_k}$.

We will show that the vector bundle map  \eqref{strangedualityatlast} is an isomorphism for $k=1$ and 
that it is projectively flat with respect to the Prym-Hitchin connections for all $k$ and all $r$.

\subsection{Classical level-rank duality at level one} In preparation for discussing flatness of level-rank duality at level one between the abelian and the higher-rank Prym varieties of an unramified double cover $\calCt \to \calC$, we recall the classical case in our set-up.
  We consider here a smooth proper relative curve $\calC \to S$ of genus $g$, and the associated relative moduli spaces $\Pic_{\calC/S}^{g-1} \to S$ of line bundles of degree $g-1$, and $\cM^s_{\SL_r} \to S$ of stable rank $r$ vector bundles with trivial determinant. Consider also the natural line bundles $\LL_{\Theta,1}= \cO(\Theta) \to \Pic_{\calC/S}^{g-1}$ and $\LL_r \to \rkdets{r}$ (the ample generator of the relative Picard group).
  \begin{theorem}\label{sd} In the above setting we have
  \begin{enumerate}[(a)]
    \item \label{sd-one} (Beauville--Narasimhan--Ramanan \cite[Theorem 3]{beauville.narasimhan.ramanan:1989}) There is a natural non-degenerate duality, unique up to $\cO_S^\times$,
      \begin{equation}\label{sd-map}\begin{tikzcd}
       \pi_\ast \LL_{\Theta,1}^r \otimes \pi_\ast \LL_r \ar[r]& \cO_S .\end{tikzcd}
      \end{equation}
    \item \label{sd-two} (Belkale \cite[Proposition 1.2]{belkale:2009}) This duality is projectively flat with respect to the Mumford--Welters and the Hitchin connection, respectively.
  \end{enumerate}
\end{theorem}
  We will detail the argument, since it will apply practically verbatim to the Prym case.
\begin{proof}  Consider the fiber-wise product $\Pic_{\calC/S}^{g-1} \times_S \rkdets{r}$ and its quotient by the natural action of $\Pic[r]$, the $r$-torsion of the Picard scheme, which is naturally identified with the relative moduli space $\cM^{*,s}_{\GL_r}$ of stable rank $r$ vector bundles of degree $r(g-1)$,
\begin{equation}\label{descent}
 \begin{tikzcd}
 \LL_{\Theta,1}^r \boxtimes \LL_r \arrow[r] \arrow[d] & \LL' \arrow[d] \\
 \Pic^{g-1} \times_S \rkdets{r} \arrow[r, swap, "{/\Pic[r]}"] & \rksstar{r}.
 \end{tikzcd} 
\end{equation}
Here, $\LL'$ has a one-dimensional space of sections corresponding to the natural Brill--Noether divisor, and its pull-back is isomorphic to $\LL^r_{\Theta,1}\boxtimes \LL_r$ by the see-saw principle. In particular, it is possible to choose a lift of the action of $\Pic[r]$ to the total space of $\LL^r_{\Theta,1}\boxtimes \LL_r$.

Now define the groups $\cG(\LL^r_{\Theta,1}),\cG(\LL_r),\cG(\LL^r_{\Theta,1}\boxtimes \LL_r)$ to be the central extensions of $\Pic[r]$ which act naturally on the total spaces of the respective line bundles. The first of these is an instance of Mumford's \emph{theta group} (\cite{mumford:1966}, \cite[Chapter 6]{birkenhake.lange:2004}), and it plays (through its representation theory) in particular a crucial role in the original construction of the Mumford--Welters connection, which we will make use of. We obtain a commutative diagram
\[
  \begin{tikzcd}
    & 1 \arrow[d] & 1 \arrow[d] & & \\
    & \GG_m \arrow[r, equals] \arrow[d, swap, "{\lambda \mapsto (\lambda,\lambda^{-1})}"] & \GG_m  \arrow[d, "{\lambda \mapsto (\lambda,\lambda^{-1})}"] & & \\
    1 \arrow[r] & \GG_m \times \GG_m \arrow[r] \arrow[d, swap, "{(\lambda,\lambda')\mapsto \lambda\lambda'}"] & \cG(\LL^r_{\Theta,1}) \underset{{\Pic[r]}}{\times} \cG(\LL_r) \arrow[r] \arrow[d] & \Pic[r] \arrow[r] \arrow[d, equals] & 0 \\
    1 \arrow[r] & \GG_m \arrow[r] \arrow[d] & \cG(\LL^r_{\Theta,1} \boxtimes \LL_r) \arrow[r] \arrow[d] & \Pic[r] \arrow[r] & 0 \\
    & 1 & 1. & &
  \end{tikzcd}
\]
As we just observed, the bottom row is split since the $\Pic[r]$-action can be lifted; choose a splitting $\sigma: \Pic[r] \to \cG(\LL^r_{\Theta,1} \boxtimes \LL)$, and use it to define a morphism $\phi: \cG(\LL^r_{\Theta,1}) \to \cG(\LL_r)$ by requiring that
\[
  \pi(g,\phi(g)) = \sigma(\pi(g)) \qquad \forall g \in \cG(\LL^r_{\Theta,1}) ,
\]
(where we denote all projections to $\Pic[r]$ by $\pi$). It is clear from the definition that $\phi(\lambda) = \lambda^{-1}$, and in particular $\phi$ is an isomorphism.

Taking the direct image of (\ref{descent}) under the projection to $S$ and using $\phi$ to act with the theta group $\cG(\LL^r_{\Theta,1})$, we therefore obtain an equivariant map of $\cG(\LL^r_{\Theta,1})$-linearized sheaves (\ref{sd-map}).
The center $\GG_m$ of $\cG(\LL^r_{\Theta,1})$ acts on $\pi_\ast\LL$ via $\lambda \mapsto \lambda^{-1}$, so this has to be a direct sum of several copies of the dual of the standard representation $\pi_\ast \cO(\LL^r_{\Theta,1})$. As the rank of $\pi_\ast \LL_r$ is known to be $r^g$ \cite{beauville.narasimhan.ramanan:1989}, there is only a single copy present, so that furthermore due to equivariance the pairing has to be non-degenerate.

As for \ref{sd-two}, first we observe (following Belkale \cite[Corollary 4.2]{belkale:2009}) that the action of the theta group $\cG(\LL^r_{\Theta,1})$ on $\pi_\ast \LL^k_r$ leaves the Hitchin connection invariant: this follows since the symbol map $\rho^{\Hit}$ is $\Pic[r]$-equivariant.  Taking the conjugate of the heat operator $D$ defining the connection with a section $\gamma$ of $\cG(\LL^r_{\Theta,1})$, we obtain a heat operator $\gamma \circ D \circ \gamma^{-1}$ with the \emph{same} symbol, so that it has to coincide with $D$ by the uniqueness statement in Theorem \ref{vgdj}.

For $k=1$, consider now the divisor $\cI$ of $\PP \pi_\ast \left( \LL^r_{\Theta,1} \boxtimes \LL_r \right)$ defined by the kernel of the map (\ref{sd-map}) and its pull-back to $\PP \pi_\ast \LL^r_{\Theta,1} \times_S \PP \pi_\ast \LL_r$, where it defines a section of $\cO(1,1)$ (up to scale). As the pairing is non-degenerate, the morphism induced by $\cI$
\begin{equation*}\begin{tikzcd}
  \PP \pi_\ast \LL^r_{\Theta,1} \ni p
  \ar[r,mapsto] &
  \cI \cap \{ p \} \times \PP \pi_\ast \LL_r \in  \left( \PP \pi_\ast \LL_r \right)^{\vee}
\end{tikzcd}\end{equation*}
is defined everywhere. By its definition, it descends from the non-zero intertwiner $\pi_\ast \LL^r_{\Theta,1} \to \left(\pi_\ast \LL_r \right)^\ast$ of irreducible (standard) representations of the theta group $\cG(\LL^r_{\Theta,1})$ given by (\ref{sd-map}). But as there is a unique projective connection compatible with this action (cf. \cite[Proposition 2.7]{welters:1983}), this proves the assertion \ref{sd-two}.
\end{proof}

 \subsection{The Prym case at level one}\label{prymsdlevel1} We can now observe how also for Prym varieties flatness of level-rank duality at level one follows from the same reasoning.

We consider here the case when $r$ is even. When $r$ is odd, we only consider the connected component
$P^{\rm even}$ on the LHS of (\ref{strangedualityatlast}). We have a diagram
\begin{equation}\label{Prym-descent}
 \begin{tikzcd}
 \calP^r_{\Xi,1}
  \boxtimes \calP_r \arrow[r] \arrow[d] & \calP' \arrow[d] \\
\left( P^{\rm even} \cup P^{\rm odd} \right) \times_S \mathcal{N}^{+,ss}_{\SL_r} \arrow[r, swap, "{/\Pr[r]}"] & \mathcal{N}^{{\rm even},+,ss}_{\GL_r},
 \end{tikzcd} 
\end{equation}
where $\Pr[r]$ denotes the group of $r$-torsion points in the connected component of the origin of $\mathrm{Nm}^{-1}(\mathcal{O})$.
We then obtain, exactly as before, an equivariant map of $\cG(r\Xi)$-linearized sheaves
\begin{equation}\label{Prym-sd-map}\begin{tikzcd}
   \pi_\ast \calP^r_{\Xi,1} \otimes \pi_\ast \calP_r \ar[r] & \pi_\ast \calP' .
\end{tikzcd}\end{equation}
\begin{theorem}[{Zelaci \cite[Theorem 6.4]{zelaci:twistconf}}] $\pi_\ast \calP' \cong \cO_S$ and the pairing (\ref{Prym-sd-map}) is non-degenerate, i.e., for any $s \in S$ the pairing induces an isomorphism
\begin{itemize}
 \item if $r$ is even, $H^0(P^{\rm even}, \calP^r_{\Xi,1})^\vee \oplus  H^0(P^{\rm odd}, \calP^r_{\Xi,1})^\vee   \stackrel{\sim} \longrightarrow H^0(\mathcal{N}^{+,ss}_{\SL_r}, \calP_r ), $
 \item if $r$ is odd, $H^0(P^{\rm even}, \calP^r_{\Xi,1})^\vee \stackrel{\sim} \longrightarrow   
 H^0(\mathcal{N}^{+,ss}_{\SL_r}, \calP_r )$.
\end{itemize}
\end{theorem}

\begin{proof} We will reuse here the notation $\mathcal{N}$ from (\ref{moduliN}).
We recall, see (\ref{iotas}), that there exists a natural map $\iota_{\SL_r}: \mathcal{N}^{+,s}_{\SL_r} \rightarrow
\mathcal{N}$, which is an isomorphism if $r$ is odd, and a double étale cover if $r$ is even. Hence, if $r$ is odd,
$H^0(\mathcal{N}^{+,ss}_{\SL_r}, \calP_r ) = H^0(\mathcal{N}^{+,s}_{\SL_r}, \calP_r ) =
H^0(\mathcal{N},\calP_r)$ and the theorem follows from \cite[Theorem 6.4]{zelaci:twistconf}. If $r$ is even, we need
to add a description of the line bundles over $\mathcal{N}$. We consider, as before, the tensor product maps
$$ t^{\rm even} : \mathcal{N} \times P^{\rm even} \rightarrow \mathcal{N}^{{\rm even},+,s}_{\GL_r} \qquad \text{and}
\qquad t^{\rm odd} : \mathcal{N} \times P^{\rm odd} \rightarrow \mathcal{N}^{{\rm even},+,s}_{\GL_r},$$ using the fact that $\iota_{\GL_r}$ is an isomorphism.
Then we observe that, by the see-saw theorem, we obtain two line bundles $\overline{\mathcal{P}_r}$
and $\overline{\mathcal{P}'_r}$ over $\mathcal{N}$ satisfying
$$
(t^{\rm even})^* \calP_{\Xi,r} = \overline{\mathcal{P}_r} \boxtimes \calP^r_{\Xi,1} \qquad \text{and}
\qquad (t^{\rm odd})^* \calP_{\Xi,r} = \overline{\mathcal{P}'_r} \boxtimes \calP^r_{\Xi,1}. $$
Repeating the argument of \cite[Theorem 6.4]{zelaci:twistconf} we obtain two isomorphisms
$$  H^0(P^{\rm even} , \calP^r_{\Xi,1})^\vee   \stackrel{\sim} \longrightarrow H^0(\mathcal{N}, \overline{\mathcal{P}_r})
 \qquad \text{and}
\qquad   H^0(P^{\rm odd}, \calP^r_{\Xi,1})^\vee \stackrel{\sim} \longrightarrow  H^0(\mathcal{N}, \overline{\mathcal{P}'_r}).$$
 Now we can conclude since $(\iota_{\SL_r})^* (\overline{\mathcal{P}_r}) = \iota_{\SL_r}^* (\overline{\mathcal{P}'_r}) =  \mathcal{P}_r$ and $(\iota_{\SL_r})_* (\mathcal{P}_r) = 
 \overline{\mathcal{P}_r} \oplus \overline{\mathcal{P}'_r}$.
\end{proof}

 \begin{corollary}\label{leveloneflat} The corresponding duality
   \[ \begin{tikzcd}
    \left( \PP \pi_{\ast} \calP^r_{\Xi,1} \right)^{\vee}  \ar[r] & \PP \pi_\ast \calP_r
     \end{tikzcd}
   \]
   is projectively flat with respect to the Prym--Hitchin and the Mumford--Welters connection, respectively.
 \end{corollary}\label{anti-lr-level1}
\begin{proof}
As the proof of Theorem \ref{sd} relies only on the uniqueness of irreducible representations of theta groups for which the center acts by $\lambda \mapsto \pm\lambda$, compatibility of the connections with the theta group actions, and uniqueness of the compatible connection on projective bundles coming from these irreducible representations, we only need to check the compatibility. But this follows by the same use of Theorem \ref{vgdj}, as the symbol $\rho^{\PH}$ is $\Pr[r]$-invariant.
\end{proof}

\section{Flatness of rank-level duality in general}\label{generalflat}
\subsection{}We have constructed the flat projective Prym-Hitchin connection on $\pi_*\calP^k_r$ in Section \ref{atlast}.  We can also consider a flat projective connection on $\pi_*\calP^r_{\Xi,k}$, as follows. If $r$ is even, we notice that the two moduli spaces 
$\mathcal{N}_{\mathrm{GL}_k}^{{\rm even},+,ss}$ and $\mathcal{N}_{\mathrm{GL}_k}^{{\rm odd},+,ss}$ are (non-canonically) isomorphic, so it will be enough to
describe the flat projective connection for one of these two components. There is a Galois cover 
$$P^{\rm even}_{\calCt/\calC}\times \mathcal{N}^{+,ss}_{\mathrm{SL}_k} \rightarrow \mathcal{N}_{\mathrm{GL}_k}^{{\rm even},+,ss}$$
given by taking tensor products.  The Galois group consists of the $k$-torsion points in $\Pr_{\calCt/\calC}$, and similar to the case considered in Section \ref{prymsdlevel1}, 
the theta group $\cG(\calP_{\Xi,1}^r)$ also acts on $\calP_{\Xi,1}^k \boxtimes \calP^r_k$.  We have $$\pi_*\calP_{\Xi,k}^r\cong \left(\pi_* \calP^{k}_{\Xi,1}\otimes \pi_*\calP^{r}_k\right)^{\cG(\calP_{\Xi,1}^r)},$$ and as  $\cG(\calP_{\Xi,1}^r)$ preserves the flat projective connections on $\pi_*\calP^{k}_{\Xi,1}$ and $\pi_*\calP^{r}_k$, this defines a flat projective connection on $\pi_*\calP_{\Xi,k}^r$.

\subsection{}We are now ready to show
\begin{theorem}\label{SDflatgeneral} The morphism $SD$ from (\ref{strangedualityatlast}) is flat for all $r$ and $k$.
\end{theorem}
\begin{proof} We begin by observing 
that the map of Lie algebras 
$$\begin{tikzcd}\mathfrak{sl}_r\oplus \mathfrak{sl}_k\ar[r,hookrightarrow]&\mathfrak{sl}_{kr}:A\oplus B\ar[r,mapsto]& A\otimes \operatorname{Id}_{k}+\operatorname{Id}_{r} \otimes B,\end{tikzcd}$$ 
(where $\otimes$ denotes the Kronecker product of matrices), which is a conformal embedding with Dynkin index $(k,r)$ (see \cite{schellekens.warner:1986}), is equivariant for the involution $d\Psi^{+}$.
Therefore, by Theorems \ref{equivconfembflat} and \ref{equivlaszlo}, we have that the morphism between \begin{equation}\label{slrk}\begin{tikzcd}\pi_*\calP_{rk}\ar[r]& \pi_*\calP_r^k \otimes \pi_*\calP_k^r\end{tikzcd}\end{equation} is flat.

We can now consider the following commutative diagram, with all maps induced by tensor products:
$$\begin{tikzcd}
& \mathcal{N}^{+,ss}_{\mathrm{SL}_r} \times \mathcal{N}^{+,ss}_{\mathrm{SL}_k} \times P^{\rm even}_{\calCt/\calC} \ar[dr] \ar[dl] & \\
\mathcal{N}^{+,ss}_{\mathrm{SL}_{kr}}  \times P^{\rm even}_{\calCt/\calC} \ar[dr]& & \mathcal{N}^{+,ss}_{\mathrm{SL}_r} \times \mathcal{N}_{\mathrm{GL}_k}^{{\rm even},+,ss} \ar[dl] \\
& \mathcal{N}_{\mathrm{GL}_{kr}}^{{\rm even},+,ss} & \\
\end{tikzcd}
$$
which gives rise to the commutative diagram of morphisms of sheaves over $S$
$$\begin{tikzcd}
& \pi_*\calP_r^k\otimes \pi_*\calP_k^r\otimes \pi_*\calP_{\Xi,1}^{kr} &
\\
\pi_*\calP_{kr} \otimes \pi_*\calP^{kr}_{\Xi,1} \ar[ur] & &  \pi_*\calP_{r}^k\otimes \pi_*\calP_{\Xi,k}^r\ar[ul]\\
& \pi_*\calP_{\Xi,kr}. \ar[ur]\ar[ul] & \\
\end{tikzcd}
$$
We need to show that the image of the tautological section of $\pi_*\calP_{\Xi,kr}$ in $ \pi_*\calP^k_r\otimes \pi_*\calP_{\Xi,k}^r$ is flat, and since 
$$\pi_*\calP^k_r\otimes \pi_*\calP_{\Xi,k}^r\cong \pi_*\calP_r^k\otimes \left(\pi_*\calP^r_k\otimes \pi_*\calP^{kr}_{
\Xi,1}\right)^{\cG(r\Xi)}$$
it suffices to show that the image in 
$\pi_*\calP^k_r\otimes \pi_*\calP^r_k\otimes \pi_*\calP^{kr}_{\Xi,1}$ is flat. But this follows from Corollary \ref{leveloneflat} 
and the flatness of (\ref{slrk}).
\end{proof}

 \appendix
\section{Twisted conformal embeddings}\label{appendixtwistedconf}
A crucial step in Belkale's proof of flatness of the strange duality morphism \cite{belkale:2009} uses the concept of \emph{conformal embeddings}.  In this appendix we review the WZW connection for twisted conformal blocks, and discuss equivariant conformal embeddings in the twisted context.  We will heavily rely on the material from \cite{belkale:2009, looijenga:2013, damiolini:2020}, to which we refer for the relevant background material.

\subsection{Overview of classical case}
Even though the mathematical version of strange duality is naturally stated in terms of non-abelian theta functions, Belkale used Laszlo's theorem \cite{laszlo:1998} to switch to the setting of bundles of conformal blocks, and then used properties of conformal embeddings there to obtain the flatness of the strange duality morphism.  

Suppose we have an embedding $\mathfrak{p}\subset\mathfrak{q}$ of complex Lie algebras, where $\mathfrak{q}$ is simple, and $\mathfrak{p}$ is semi-simple, with a decomposition in simple summands $\mathfrak{p}=\oplus_{i=1}^k \mathfrak{p}_i$.  Associated with this is a \emph{Dynkin index} $d=(d_1,\dots,d_k)\in \mathbb{Z}_{>0}^k$.  If we denote the associated affine Lie algebras as $\widehat{\mathfrak{q}}$ and $\widehat{\mathfrak{p}}_i$, then we have an induced map $$\widehat{\mathfrak{p}}=\oplus_{i=1}^k \widehat{\mathfrak{p}}_i\rightarrow \widehat{\mathfrak{q}},$$ such that the generating central element $c_i$ of $\widehat{\mathfrak{p}}_i$ gets mapped to $d_i$ times the generating central element of $\widehat{\mathfrak{p}}$.

The embedding $\mathfrak{p}\subset\mathfrak{q}$ is said to be \emph{conformal} if $a_{\widehat{\mathfrak{q}}}=a_{\widehat{\mathfrak{p}}},$ where
$$a_{\widehat{\mathfrak{q}}}=\frac{\dim \mathfrak{q}}{h^{\vee}_{\mathfrak{q}}+1}$$ is the central charge of $\widehat{\mathfrak{q}}$ at level 1 (the only level at which this can occur), and $$a_{\widehat{\mathfrak{p}}}=\sum_{i=1}^k \frac{d_i \dim \mathfrak{p}_i}{h^{\vee}_{\mathfrak{p}_i}+d_i}$$ is the sum of the central charges of the $\mathfrak{p}_i$ at level $d_i$ (here $h^{\vee}_{\mathfrak{q}}$ and $h^{\vee}_{\mathfrak{p}_i}$ are the respective dual Coxeter numbers).  The relevance of this is that it ensures compatibility of the corresponding Segal-Sugawara constructions (see \cite[Section 5.2]{belkale:2009}).  In particular, we can consider the so-called co-set representation of the two Segal-Sugawara constructions, which is a representation of the Virasoro algebra which will have, if the embedding is conformal, trivial central charge.  This guarantees that the whole coset representation is trivial, which is equivalent to saying that
the diagram
 \begin{equation}\label{segsugcomp}
\begin{tikzcd} 
& &\overline{U}\widehat{\mathfrak{p}}\Bigg[\frac{1}{c_{\mathfrak{p}_i}+h^{\vee}_{\mathfrak{p}_i}}\Bigg]^{\operatorname{Aut}(\mathfrak{p})}
\ar[dd]
\\ 
\operatorname{Vir}
\ar[urr,"T_{\mathfrak{p}}"]\ar[drr,"T_{\mathfrak{q}}" ']& 
\\ 
& &  \overline{U}\widehat{\mathfrak{q}}\Bigg[\frac{1}{c_{\mathfrak{q}}+h^{\vee}_{\mathfrak{q}}}\Bigg]^{\operatorname{Aut}(\mathfrak{q})}
\end{tikzcd}\end{equation}
commutes.  Here $\operatorname{Vir}$ is the Virasoro algebra, $\overline{U}\widehat{\mathfrak{p}}$ and $\overline{U}\widehat{\mathfrak{q}}$ are the suitable completions of the universal enveloping algebras of $\widehat{\mathfrak{p}}$ and $\widehat{\mathfrak{q}}$, $c_{\mathfrak{p}_i}$ and $c_{\mathfrak{q}}$ are the central generators of $\widehat{\mathfrak{p}}_i$ and $\widehat{\mathfrak{q}}$, $T_{\mathfrak{p}}$ and $T_{\mathfrak{q}}$ are the Segal-Sugawara morphisms, and the vertical arrow is naturally induced by the inclusion $\mathfrak{p}\subset\mathfrak{q}$.

The Segal-Sugawara construction is the crucial ingredient in the WZW/TUY connection on bundles on conformal blocks, and the main corollary of  (\ref{segsugcomp}) is

\begin{theorem}[{\cite[Proposition 5.8]{belkale:2009}}]
If $\mathfrak{p}\subset\mathfrak{q}$ is a conformal embedding of Lie algebras (with $\mathfrak{q}$ simple, and $\mathfrak{p}$ semi-simple as before), then the induced map between the projective bundles of conformal blocks is flat.
\end{theorem}

\subsection{Twisted conformal blocks}\label{twistconfoverview}
We now want to indicate how this result also holds for bundles of twisted conformal blocks.  As in Section \ref{twistedlaszlo}, we will use Damiolini's approach to the latter \cite{damiolini:2020}, which requires the Galois group $\Gamma$ to be cyclic of prime order.  Though the result is no doubt true in greater generality, this will suffice for our purposes.   
So from now on the Lie algebras $\mathfrak{p}_i$ and $\mathfrak{q}$ (respectively semi-simple and simple, as before) will come equipped with the action of a finite cyclic group $\Gamma$ of prime order, such that the embedding $\mathfrak{p}\subset\mathfrak{q}$ is $\Gamma$-equivariant.  

In order to define the bundles of conformal blocks, the choice of (at least) one marked point is required.  We will denote by 
$\mathcal{H}ur(\Gamma)_{g,1}$ the Hurwitz stack of \'etale Galois covers $\calCt\rightarrow \calC$ with Galois group $\Gamma$, equipped with the choice of a marked point on $\calC$.  
(In \cite{damiolini:2020} also ramication, and nodal curves, are considered, but this will not be a concern for us.)  Given a level $\ell \in \mathbb{Z}_{>0}$, 
the bundle of conformal blocks  $\mathbb{V}_{\mathfrak{h},\ell, \calCt\rightarrow \calC}^{\dagger}$, is constructed by Damiolini 
on $\mathcal{H}ur(\Gamma)_{g,1}$ (we will only need the case corresponding to the trivial representations of $\mathfrak{p}$ and $\mathfrak{q})$. Similar to \cite[Section 5.1]{belkale:2009}, this construction naturally generalizes to our equivariant semi-simple case, to give bundles 
$\mathbb{V}_{\mathfrak{p},(\ell_1,\dots,\ell_k), \calCt\rightarrow \calC}^{\dagger}$.

If $P$ and $Q$ are the simply-connected semi-simple groups corresponding to the Lie algebras $\mathfrak{p}$ and $\mathfrak{q}$ (which also carry an action of $\Gamma$), we can consider the parahoric Bruhat-Tits group schemes $\mathcal{P}$ and $\mathcal{Q}$ associated to this data as invariants of the Weil restriction of the constant group schemes on $\calCt$ (cfr. \cite[p. 500]{heinloth:2010}). We will denote the associated relative moduli stacks of torsors as $\pi^{\mathcal{P}}:\mathcal{B}un_{\mathcal{P}}\rightarrow \mathcal{H}ur(\Gamma)_{g,1}$ and $\pi^{\mathcal{Q}}:\mathcal{B}un_{\mathcal{Q}}\rightarrow \mathcal{H}ur(\Gamma)_{g,1}$ (no choice of marked point is needed here, so these are actually pull-backs from the forgetful morphism to the unmarked Hurwitz stacks). Remark that we have a natural morphism of group schemes $\mathcal{P}\rightarrow\mathcal{Q}$, and hence extension of structure group gives a morphism $\mathcal{B}un_{\mathcal{P}}\rightarrow \mathcal{B}un_{\mathcal{Q}}$.  The relative Picard group of $\mathcal{B}un_{\mathcal{Q}}$ is cyclic with generator $\mathcal{L}$ \cite{heinloth:2010}, and given $\ell\in \mathbb{Z}_{>0}$, the line bundle $\mathcal{L}^{\otimes\ell}$ pulls back to a line bundle $\mathcal{L}_{(\ell d_1,\dots,\ell d_k)}$ on $\mathcal{B}un_{\mathcal{P}}$.

The construction of the bundles of twisted conformal blocks on $\mathcal{H}ur_{g,1}$ goes in the following steps (we referred to \cite[Section 3]{damiolini:2020} for details): 
\begin{enumerate}[label=(\roman*)]
\item Take the associated bundles of Lie algebras for $\mathcal{P}$  and $\mathcal{Q}$, which we shall denote respectively by $\mathfrak{h}^{\mathcal{P}}$ and $\mathfrak{h}^{\mathcal{Q}}$, and restrict these, firstly to the punctured curve, to obtain bundles $\mathfrak{h}^{\mathcal{P}}_{\mathcal{A}}$ and $\mathfrak{h}^{\mathcal{Q}}_{\mathcal{A}}$, and secondly to the punctured formal neighbourhood $\mathbb{D}^{\degree}\rightarrow S$ of the marked point, to obtain bundles $\mathfrak{h}^{\mathcal{P}}_{\mathbb{D}^{\degree}}$ and $\mathfrak{h}^{\mathcal{Q}}_{\mathbb{D}^{\degree}}$ (at a point these will look like the loop algebras of $\mathfrak{p}$ and $\mathfrak{q}$).
\item There are canonical central extensions of the latter, denoted by respectively $\widehat{\mathfrak{h}}^{\mathcal{P}}_{\mathbb{D}^{\degree}}$ and $\widehat{\mathfrak{h}}^{\mathcal{Q}}_{\mathbb{D}^{\degree}}$, with the rank of the center of the former given by $k$.  By the residue theorem, $\mathfrak{h}^{\mathcal{P}}_{\mathcal{A}}$ and $\mathfrak{h}^{\mathcal{Q}}_{\mathcal{A}}$ are naturally sub-bundles of $\widehat{\mathfrak{h}}^{\mathcal{P}}_{\mathbb{D}^{\degree}}$ and $\widehat{\mathfrak{h}}^{\mathcal{Q}}_{\mathbb{D}^{\degree}}$.
\item\label{vermquot} Given $\ell\in \mathbb{Z}_{>0}$, or $(\ell_1, \dots, \ell_k)\in \mathbb{Z}^k_{>0}$, define the Verma modules of that level, as $$\widetilde{\mathcal{H}}^{\mathcal{Q}}_{\ell} = U\widehat{\mathfrak{h}}_{\mathbb{D}^{\degree}}^{\mathcal{Q}}\Big/ \left(U\widehat{\mathfrak{h}}_{\mathbb{D}^{\degree}}^{\mathcal{Q}}\circ F^0\mathfrak{h}_{\mathbb{D}^{\degree}}^{\mathcal{Q}}, c=\ell\right)$$
and 
$$\widetilde{\mathcal{H}}^{\mathcal{P}}_{(\ell_1,\dots,\ell_k)}=U\widehat{\mathfrak{h}}_{\mathbb{D}^{\degree}}^{\mathcal{P}}\Big/ \left(U\widehat{\mathfrak{h}}_{\mathbb{D}^{\degree}}^{\mathcal{P}}\circ F^0\mathfrak{h}_{\mathbb{D}^{\degree}}^{\mathcal{P}}, c_i=\ell_i\right).$$  Here $F^0\mathfrak{h}_{\mathbb{D}^{\degree}}^{\mathcal{P}}$ and $F^0\mathfrak{h}_{\mathbb{D}^{\degree}}^{\mathcal{Q}}$ denote the subsheaves of sections that extend to the formal disk.  Both of $\widetilde{\mathcal{H}}^{\mathcal{Q}}_{\ell}$ and $\widetilde{\mathcal{H}}^{\mathcal{P}}_{(\ell_1,\dots, \ell_k)}$ have unique maximal proper submodules, and the respective quotients are denoted by ${\mathcal{H}}^{\mathcal{Q}}_{\ell}$ and 
${\mathcal{H}}^{\mathcal{P}}_{(\ell_1, \dots, \ell_k)}.$ 
\item Finally the bundles of covacua are defined as $$\mathbb{V}_{\mathfrak{h}^{\mathcal{Q}},\ell, \calCt\rightarrow \calC}=\left(\mathfrak{h}_{\mathcal{A}}^{\mathcal{Q}}\circ {\mathcal{H}}^{\mathcal{Q}}_{\ell}\right)\Big\backslash {\mathcal{H}}^{\mathcal{Q}}_{\ell}$$
and $$\mathbb{V}_{\mathfrak{h}^{\mathcal{P}},(\ell_1,\dots,\ell_k), \calCt\rightarrow \calC}=\left(\mathfrak{h}_{\mathcal{A}}^{\mathcal{P}}\circ {\mathcal{H}}^{\mathcal{P}}_{(\ell_1,\dots,\ell_k)}\right)\Big\backslash {\mathcal{H}}^{\mathcal{P}}_{(\ell_1,\dots,\ell_k)},$$
and the bundles of conformal blocks $\mathbb{V}_{\mathfrak{h}^{\mathcal{Q}},\ell, \calCt\rightarrow \calC}^{\dagger}$ and $\mathbb{V}_{\mathfrak{h}^{\mathcal{P}},(\ell_1,\dots,\ell_k), \calCt\rightarrow \calC}^{\dagger}$ are their duals.
\end{enumerate}

We firstly observe (without requiring the embedding $\mathfrak{p}\subset\mathfrak{q}$ to be conformal) that we obtain a natural morphism between the bundles of twisted conformal blocks.
\begin{proposition}
In the setting above, there is a natural diagram 
\begin{equation*}
\begin{tikzcd}
\pi^{\mathcal{P}}_*\mathcal{L}^{\otimes\ell}\ar[d]\ar[r, "\cong"] & \mathbb{V}_{\mathfrak{h}^{\mathcal{P}},\ell, \calCt\rightarrow \calC}^{\dagger} \ar[d]\\
\pi^{\mathcal{Q}}_*\mathcal{L}_{(\ell d_1,\dots, \ell d_k)}\ar[r, "\cong" '] & 
\mathbb{V}_{\mathfrak{h}^{\mathcal{Q}},(\ell d_1,\dots,\ell d_k), \calCt\rightarrow \calC}^{\dagger}
\end{tikzcd}
\end{equation*}
that commutes up to $\mathcal{O}_S^{\times}$.
\end{proposition}
Here the horizontal morphisms are given by the correspondence between non-abelian theta functions and twisted conformal blocks \cite[Theorem 12.1]{hong.kumar:2018}.  We will work throughout with families of $\Gamma$-covers $\calCt\rightarrow \calC\rightarrow S$ as above, parametrised by a scheme $S$.
\begin{proof} 
To obtain the right vertical arrow, it suffices to remark that the $\Gamma$-equivariance of $\mathfrak{p}\subset \mathfrak{q}$ gives natural morphisms $\mathcal{P}\rightarrow \mathcal{Q}$, $\mathfrak{h}^{\mathcal{P}}\rightarrow \mathfrak{h}^{\mathcal{Q}}$, $\mathfrak{h}^{\mathcal{P}}_{\mathcal{A}} \rightarrow \mathfrak{h}^{\mathcal{Q}}_{\mathcal{A}}$, $\mathfrak{h}^{\mathcal{P}}_{\mathbb{D}^{\degree}} \rightarrow\mathfrak{h}^{\mathcal{Q}}_{\mathbb{D}^{\degree}}$; as well as (by scaling the central generators using the Dynkin index $d$) $\widehat{\mathfrak{h}}^{\mathcal{P}}_{\mathbb{D}^{\degree}} \rightarrow\widehat{\mathfrak{h}}^{\mathcal{Q}}_{\mathbb{D}^{\degree}}$.  These in turn induce morphisms
$ \widetilde{\mathcal{H}}^{\mathcal{P}}_{(\ell_1,\dots,\ell_k)}  \rightarrow \widetilde{\mathcal{H}}^{\mathcal{Q}}_{\ell}$, $ {\mathcal{H}}^{\mathcal{P}}_{(\ell_1,\dots,\ell_k)}\rightarrow{\mathcal{H}}^{\mathcal{Q}}_{\ell}$ and $\mathbb{V}_{\mathfrak{h}^{\mathcal{P}},(\ell_1,\dots,\ell_k), \calCt\rightarrow \calC}\rightarrow \mathbb{V}_{\mathfrak{h}^{\mathcal{Q}},\ell, \calCt\rightarrow \calC}$.
Using \cite[Theorem 12.1]{hong.kumar:2018}, the rest of the statement follows similarly to \cite[Proposition 5.2]{belkale:2009}.
\end{proof}

\subsection{Overview of construction of twisted WZW connection}\label{twistedWZWoverview}
Following Looijenga's approach in the non-twisted case \cite{looijenga:2013}, the construction of the twisted WZW connection by Damiolini \cite{damiolini:2020} (and extended to the semi-simple case) proceeds as follows:
\begin{enumerate}[label=(\roman*)]
\item  If $\theta_{\mathbb{D}^{\degree}/S}$ denotes the sheaf of Lie algebras on $S$ given as vertical vector fields on the formal punctured neighbourhood on $C$, there exists a canonical central extension
$$\begin{tikzcd}
0\ar[r] & \hbar\mathcal{O}_S\ar[r] & \widehat{\theta}_{\mathbb{D}^{\degree}/S}\ar[r] & \theta_{\mathbb{D}^{\degree}/S}\ar[r] & 0,\end{tikzcd}$$ referred to as the Virasoro algebra over $S$.  There is also a canonical short exact sequence
$$\begin{tikzcd}0\ar[r] & \theta_{\mathbb{D}^{\degree}/S}\ar[r]& \theta_{\mathbb{D}^{\degree},S}\ar[r]& T_S\ar[r]& 0,
\end{tikzcd}$$
where $\theta_{\mathbb{D}^{\degree},S}$ is the sheaf of vector fields on all of the formal punctured neighbourhood on $C$.

\item \label{twistedsegsug} The twisted Segal-Sugawara construction now gives \cite[Proposition 4.17]{damiolini:2020} morphisms \begin{equation}\label{segsugtwist}
T_{\mathfrak{h}^\mathcal{P}}:\widehat{\theta}_{\mathbb{D}^{\degree}/S}\rightarrow \overline{U}\widehat{\mathfrak{h}}^{\mathcal{P}}_{\mathbb{D}^{\degree}}\Bigg[\frac{1}{c_{\mathfrak{p}_i}+h^{\vee}_{\mathfrak{p}_i}}\Bigg]\hspace{.5cm}
\text{and} 
\hspace{.5cm} T_{\mathfrak{h}^\mathcal{Q}}:\widehat{\theta}_{\mathbb{D}^{\degree}/S}\rightarrow \overline{U}\widehat{\mathfrak{h}}^{\mathcal{Q}}_{\mathbb{D}^{\degree}}\Bigg[\frac{1}{c_{\mathfrak{q}}+h^{\vee}_{\mathfrak{q}}}\Bigg],\end{equation} where $\overline{U}$ denotes the completion of the universal enveloping algebras of $\widehat{\mathfrak{h}}^{\mathcal{P}}_{\mathbb{D}^{\degree}}$ and $\widehat{\mathfrak{h}}^{\mathcal{Q}}_{\mathbb{D}^{\degree}}$ with respect to suitable filtration $F^*\widehat{\mathfrak{h}}^{\mathcal{P}}_{\mathbb{D}^{\degree}}$ and $F^*\widehat{\mathfrak{h}}^{\mathcal{Q}}_{\mathbb{D}^{\degree}}$.

\item \label{segsugdescent} The morphisms $T_{\mathfrak{h}^\mathcal{P}}$ and $T_{\mathfrak{h}^\mathcal{Q}}$ induce
morphisms (by abuse of notation denoted in the same way)
\begin{equation*}
T_{\mathfrak{h}^\mathcal{P}}:\widehat{\theta}_{\mathbb{D}^{\degree}/S}\rightarrow 
\operatorname{End}\left(\mathcal{F}^+\left(\mathfrak{h}^{\mathcal{P}}_{\mathbb{D}^{\degree}}\right)\right)
\hspace{.5cm}
\text{and} 
\hspace{.5cm} T_{\mathfrak{h}^\mathcal{Q}}:\widehat{\theta}_{\mathbb{D}^{\degree}/S}\rightarrow \operatorname{End}\left(\mathcal{F}^+\left(\mathfrak{h}^{\mathcal{Q}}_{\mathbb{D}^{\degree}}\right)\right),\end{equation*}
where $$\mathcal{F}^+\left(\mathfrak{h}^{\mathcal{Q}}_{\mathbb{D}^{\degree}}\right)=\left(\overline{U}\widehat{\mathfrak{h}}^{\mathcal{Q}}_{\mathbb{D}^{\degree}}\Big/ \overline{U}\widehat{\mathfrak{h}}^{\mathcal{Q}}_{\mathbb{D}^{\degree}}\circ F^1\widehat{\mathfrak{h}}^{\mathcal{Q}}_{\mathbb{D}^{\degree}}\right)\Bigg[\frac{1}{c_{\mathfrak{q}}+h^{\vee}_{\mathfrak{q}}}\Bigg]$$
and 
$$\mathcal{F}^+\left(\mathfrak{h}^{\mathcal{P}}_{\mathbb{D}^{\degree}}\right)=\left(\overline{U}\widehat{\mathfrak{h}}^{\mathcal{P}}_{\mathbb{D}^{\degree}}\Big/ \overline{U}\widehat{\mathfrak{h}}^{\mathcal{P}}_{\mathbb{D}^{\degree}}\circ F^1\widehat{\mathfrak{h}}^{\mathcal{P}}_{\mathbb{D}^{\degree}}\right)\Bigg[\frac{1}{c_{\mathfrak{p}_i}+h^{\vee}_{\mathfrak{p}_i}}\Bigg].$$

\item\label{segsugproj} The projectivizations of $T_{\mathfrak{h}^\mathcal{P}}$ and  $T_{\mathfrak{h}^\mathcal{Q}}$ combine with the subsheaf $F^0\theta_{\mathbb{D}^{\degree},S}\subset\theta_{\mathbb{D}^{\degree},S}$ (acting via coefficient-wise derivation) to give morphisms of Lie algebras $$\begin{tikzcd}\mathbb{P}T_{\mathfrak{h}^{\mathcal{P}},S}:\theta_{\mathbb{D}^{\degree},S}\ar[r]& \mathcal{D}^{(1)}_S\left(\mathcal{H}_{(\ell_1,\dots,\ell_k)}^{\mathcal{P}}\right)\Big/\mathcal{O}_S \end{tikzcd}$$ and $$\begin{tikzcd}\mathbb{P}T_{\mathfrak{h}^{\mathcal{Q}},S}:\theta_{\mathbb{D}^{\degree},S}\ar[r]& \mathcal{D}^{(1)}_S\left(\mathcal{H}_{\ell}^{\mathcal{Q}}\right)\Big/\mathcal{O}_S.\end{tikzcd}$$

\item\label{finalquot} In turn, $\mathbb{P}T_{\mathfrak{h}^{\mathcal{P}}}$ and $\mathbb{P}T_{\mathfrak{h}^{\mathcal{Q}}}$ induce the flat projective connections on the bundles of covacua (\cite[Lemma 4.22]{damiolini:2020})
$$\begin{tikzcd}\nabla^{\mathcal{P}}:T_S\ar[r]& \mathcal{A}
\left(\mathbb{V}_{\mathfrak{h}^{\mathcal{P}},(\ell_1,\dots,\ell_k), \calCt\rightarrow \calC}\right)\Big/\mathcal{O}_S \end{tikzcd}$$ 
and 
$$\begin{tikzcd}\nabla^{\mathcal{Q}}:T_S\ar[r]& \mathcal{A}
\left(\mathbb{V}_{\mathfrak{h}^{\mathcal{Q}},\ell, \calCt\rightarrow \calC}\right)\Big/\mathcal{O}_S.\end{tikzcd}$$

\end{enumerate}
\begin{remark}
Though we have no need to do so, it is possible to avoid going to projectivizations, to obtain a \emph{$\lambda$-flat connection} (in the terminology of \cite{looijenga:2013}), also known as a \emph{twisted $\mathcal{D}$-module}, as is done in \cite[\S 12]{deshpande.mukhopadhyay:2023}.  Twisted $\mathcal{D}$-modules induce flat projective connections, but carry a bit more information.
\end{remark}

\subsection{Conformal embeddings and twisted flatness}
We will now further require the equivariant embedding $\mathfrak{p}\subset\mathfrak{q}$ to be conformal as before (and hence put $\ell=1$), and consider the flat projective WZW connections on $\mathbb{V}_{\mathfrak{q},\ell, \calCt\rightarrow \calC}^{\dagger}$ and $\mathbb{V}_{\mathfrak{p},(\ell d_1,\dots,\ell d_k), \calCt\rightarrow \calC}^{\dagger}$ as constructed in \cite[Section 4]{damiolini:2020}.

\begin{theorem}\label{equivconfembflat}If the $\Gamma$-equivariant embedding $\mathfrak{p}\rightarrow \mathfrak{q}$ is conformal, the natural morphism between the projective bundles of twisted conformal blocks  $\mathbb{V}_{\mathfrak{q},1, \calCt\rightarrow \calC}^{\dagger}$ and \linebreak $\mathbb{V}_{\mathfrak{p},(d_1,\dots,d_k), \calCt\rightarrow \calC}^{\dagger}$ preserves the projective connections.
\end{theorem}
\begin{proof}
If $\mathfrak{p}\subset\mathfrak{q}$ is conformal, and we take $\ell=1, (\ell_1,\dots,\ell_k)=(d_1,\dots,d_k)$, the Segal-Sugawara morphisms (\ref{segsugtwist}) will again form a commutative diagram, as in (\ref{segsugcomp}),  in the construction of the connection outlined in \ref{twistedWZWoverview}. 
 (As in \cite{belkale:2009}, this follows from considering the coset representation for $T_{\mathfrak{h}^\mathcal{P}}$ and  $T_{\mathfrak{h}^\mathcal{Q}}$, which by \cite[Proposition 3.2(c)]{kac.wakimoto:1988} is trivial for each closed point in $S$, hence trivial on all of $S$.)

 Together with the natural map $$\phi:\begin{tikzcd}\mathcal{F}^+\left(\mathfrak{h}^{\mathcal{P}}_{\mathbb{D}^{\degree}}\right)\ar[r] & 
\mathcal{F}^+\left(\mathfrak{h}^{\mathcal{Q}}_{\mathbb{D}^{\degree}}\right)\end{tikzcd}$$ this gives in step \ref{segsugdescent} $$\phi\left(T_{\mathfrak{h}^\mathcal{Q}}(\widehat{X})(s)\right)=T_{\mathfrak{h}^\mathcal{P}}(\widehat{X})(\phi(s))$$ for $\widehat{X}\in \widehat{\theta}_{\mathbb{D}^{\degree}/S}$.

 The result then follows by following this through steps \ref{segsugdescent} and \ref{finalquot}, and dualizing to the bundles of conformal blocks.
\end{proof}

\def\cftil#1{\ifmmode\setbox7\hbox{$\accent"5E#1$}\else
  \setbox7\hbox{\accent"5E#1}\penalty 10000\relax\fi\raise 1\ht7
  \hbox{\lower1.15ex\hbox to 1\wd7{\hss\accent"7E\hss}}\penalty 10000
  \hskip-1\wd7\penalty 10000\box7}
  \def\cftil#1{\ifmmode\setbox7\hbox{$\accent"5E#1$}\else
  \setbox7\hbox{\accent"5E#1}\penalty 10000\relax\fi\raise 1\ht7
  \hbox{\lower1.15ex\hbox to 1\wd7{\hss\accent"7E\hss}}\penalty 10000
  \hskip-1\wd7\penalty 10000\box7}
  \def\cftil#1{\ifmmode\setbox7\hbox{$\accent"5E#1$}\else
  \setbox7\hbox{\accent"5E#1}\penalty 10000\relax\fi\raise 1\ht7
  \hbox{\lower1.15ex\hbox to 1\wd7{\hss\accent"7E\hss}}\penalty 10000
  \hskip-1\wd7\penalty 10000\box7}
  \def\cftil#1{\ifmmode\setbox7\hbox{$\accent"5E#1$}\else
  \setbox7\hbox{\accent"5E#1}\penalty 10000\relax\fi\raise 1\ht7
  \hbox{\lower1.15ex\hbox to 1\wd7{\hss\accent"7E\hss}}\penalty 10000
  \hskip-1\wd7\penalty 10000\box7} \def\cprime{$'$}

\end{document}